\documentclass[11pt]{article}
\usepackage{amsmath,amsfonts,amsthm,epsfig,latexsym,graphicx}

 \def\NN{{\mathbb N}}  

 \def\RR{{\mathbb R}}

 \def\ZZ{{\mathbb Z}}

\def\Si{\Sigma}
\def\La{\Lambda}

\def\Ga{\Gamma}

    \def\cT{{\cal T}}
\def\cC{{\cal C}}    \def\cU{{\cal U}}
   \def\cP{{\cal P}} \def\cV{{\cal V}}
    
\def\cF{{\cal F}}  \def\cL{{\cal L}} \def\cR{{\cal R}}

\def\interior{\operatorname{Interior}}

\newtheorem*{whitney}{Whitney's extension theorem}
\newtheorem*{maintheo}{Main Theorem}
\newtheorem{theo}{Theorem}
\newtheorem*{theo*}{Theorem}

\newtheorem{lemm}{Lemma}[section]
\newtheorem{sublemm}[lemm]{Sublemma}
\newtheorem{claim}[lemm]{Claim}
\newtheorem{coro}[lemm]{Corollary}
\newtheorem{prop}[lemm]{Proposition}
\newtheorem{defi}[lemm]{Definition}

\newtheorem{ques}[lemm]{Question}
\newtheorem{addendum}[lemm]{Addendum}

\theoremstyle{definition}
\newtheorem{rema}[lemm]{Remark}
\newtheorem{rems}[lemm]{Remarks}

\title{Center manifolds for partially hyperbolic set without strong unstable connections}
\author{\setcounter{footnote}{-1}
Christian Bonatti\footnote{Partially supported by the ANR project \emph{DynNonHyp} BLAN08-2 313375
and by the \emph{Balzan Research Project} of J.Palis.
The authors acknowledge the IFUM and the CMAT (Montevideo) where part of this text was written.} and Sylvain Crovisier\footnotemark[0]}

\date{\today}

\begin{document}
\maketitle

\begin{abstract}
We consider compact sets which are invariant and partially hyperbolic under the dynamics
of a diffeomorphism of a manifold.
We prove that such a set $K$ is contained in a locally invariant center submanifold if and only if each strong stable and strong unstable leaf intersects $K$ at exactly one point.
\end{abstract}

\section{Introduction}
From the first works on dynamical systems, one has tried to reduce 
the dimension of the system and many technics have been developed 
for that purpose: first integrals (for some conservative systems), Poincar\'e return 
maps on transverse sections (for flows), quotients by invariant foliations, etc...
One of these technics is the famous 
``center manifold theorem", see for instance~\cite{HPS,Wi}. Consider a fixed point $x$ 
of a diffeomorphism $f$ on a manifold $M$ such that the 
differential $Df(x)$ leaves invariant a splitting $T_xM= E^{ss}\oplus E^c\oplus E^{uu}$, 
corresponding to the parts of the spectrum of $Df(x)$ whose moduli are, respectively, 
strictly less than one, equal to one, and strictly greater than one. Then there exists a 
locally $f$-invariant manifold $W^c$ through $x$ tangent at $x$ to $E^c$.
Furthermore the local topological dynamics of $f$ is the product of 
the restriction $f|_{W^c}$ by a uniform contraction in the $E^{ss}$ direction 
and by a uniform dilation in the $E^{uu}$-direction. 

\subsection{Main result}
We propose here some generalization of the center manifold theorem where 
the fixed point $x$ is replaced by an invariant compact set. Recal that a $Df$-invariant splitting $TM=E^c\oplus E^{uu}$
defined over an invariant compact set $K$ is \emph{partially hyperbolic} with \emph{strong unstable direction $E^{uu}$}
if the vectors in $E^{uu}$ are uniformly expanded and the possible expansion in $E^c$ is strictly weaker than the expansion in $E^{uu}$.
More precisely there are constants $\lambda_K>1$
and $C>0$ such that for each $x\in K$, each unit vectors
$u\in E^{c}(x)$, $v\in E^{uu}(x)$ and each $n\geq 1$, one has:
$$\|Df^n(x)\cdot v\|> C\lambda_K^n\quad \text{and} \quad
\|Df^n(x)\cdot v\|> C \lambda_K^n \|Df^n(x)\cdot u\|.$$

With these hypotheses, any point in $K$ has a well defined \emph{strong unstable manifold} $W^{uu}(x)$ tangent to $E^{uu}(x)$.
It is the set of points whose backward iterates get closer to those of $x$ at almost the same rates as the contraction of the vectors
in $E^{uu}$ by the backward iterates of $Df$.

\begin{maintheo} Let $f$ be a $C^1$-diffeomorphism and $K$ a partially hyperbolic compact invariant set  such that
$TM|_K=E^c\oplus E^{uu}$. Then, the next two properties are equivalent.
\begin{enumerate}
\item There exists a compact $C^1$-submanifold $S$ with boundary which:
\begin{itemize}
\item[--] contains $K$ in its interior,
\item[--] is tangent to $E^c$ at each point of $K$ (i.e. $T_xS=E^c(x)$ for each $x\in K$),
\item[--] is locally invariant: $f(S)\cap S$ contains a neighborhood of $K$ in $S$,
\end{itemize}
\item The strong unstable manifold of any $x\in K$ intersect $K$ only at $x$ (i.e. $W^{uu}(x)\cap K=\{x\}$).
\end{enumerate}
\end{maintheo}
\medskip

\begin{rems}
\begin{enumerate}
\item The submanifold $S$ is in general not unique: if $S'$ is another submanifold
the intersection $S\cap S$ could be reduced to $K$, as for the center manifolds of a fixed point. 
\item The implication $1\Rightarrow 2$ in the Main Theorem is immediate: if $K$ is contained in $S$
and if $x,y\in K$ share the same strong unstable manifold, the points
$x'=f^{-n}(x), y'=f^{-n}(y)$ for $n\geq 0$ large belong to a same local strong unstable
manifold. The transversality between $T_xS=E^c(x)$ and $T_xW^{uu}_x=E^{uu}(x)$
implies that $x'=y'$, hence that $x=y$.
\item After we wrote a first version of this text, Genevi\`eve Raugel mentioned us that
center manifolds for partially hyperbolic invariant sets
have been also built before by Chow-Liu-Yi~\cite{CLY}
for flows generated by a vector fields under different assumptions:
they require that the set is tangent at each point to its center bundle
and that its geometry is ``bounded'' (admissibility condition).
Our result shows that their second assumption is not necessary and
that the first one can be replaced by a dynamical property on the strong unstable
lamination, which is easier to check in practice. Our assumptions are optimal since we get
an equivalence.
\end{enumerate}
\end{rems}
\bigskip

The compact $K$ is not necessarily the maximal invariant set in a neighborhood. The Corollary below extends the conclusion of the
theorem to the maximal invariant set of a neighborhood. 

\begin{coro}\label{c.principal} Under the conclusion of the Main Theorem,  there is a neighborhood $U$ of $K$ such that the maximal
invariant set $\La$ in $U$ is contained in the interior of $S$. 
\end{coro}

We consider now an invariant compact set $K$ having a $Df$-invariant \emph{partially hyperbolic splitting} in three bundles
$TM|_K=E^{ss}\oplus E^c\oplus E^{uu}$, with strong stable,  center, and strong unstable directions $E^{ss}$, $E^c$, and $E^{uu}$,
respectively\footnote{That is, the splittings in two bundles  $TM|_K= (E^{ss}\oplus E^c)\oplus E^{uu}$ and
$TM|_K= (E^{uu}\oplus E^c)\oplus E^{ss}$ are partially hyperbolic for $f$ and $f^{-1}$, respectively,
(with strong unstable direction $E^{uu}$ and $E^{ss}$ respectively)}.
The strong stable manifold $W^{ss}(x)$ of the points of $K$ are the strong unstable manifolds, tangent to $E^{ss}(x)$, for $f^{-1}$.

\begin{coro}\label{c.saddle} Let $K$ be a compact invariant set of a diffeomorphism $f$, admiting a partially hyperbolic splitting
$TM|_K=E^{ss}\oplus E^c\oplus E^{uu}$.
Then there is a compact $C^1$-submanifold $S$ with boundary which contains $K$ in its interior, is tangent to $E^c$ at each point of $K$
and is locally invariant, if and only if the strong stable and strong unstable manifolds of any point $x\in K$ intersect $K$ only at $x$
(i.e. $W^{uu}(x)\cap K=\{x\}=W^{ss}(x)\cap K$).
\end{coro}

We will sometimes reformulate the assumptions on $K$
with the following terminology.
We will say that a partially hyperbolic set $K$ admits a \emph{strong stable} or a \emph{strong unstable connection}
if there is $x\in K$ such that $W^{ss}(x)$ or $W^{uu}(x)$, respectively,  meets $K$ in a point different from $x$.
A compact invariant set  $K$, endowed with a partially hyperbolic structure of type $E^{ss}\oplus E^c$,
$E^{c}\oplus E^{uu}$ or $E^{ss}\oplus E^c\oplus E^{uu}$,  has \emph{no strong connection} if it has no strong stable connection, no
strong unstable connection, or no strong stable nor strong unstable connection, respectively.
\medskip

In the previous statements, the locally invariant submanifold $S$ is tangent to the center direction so that
it is \emph{normaly hyperbolic}.
In particular it persists by small perturbations:

\begin{coro}\label{c.robustness}
Under the conclusion of the Main Theorem, there exists a submanifold with boundary
$S'\subset S\cap f(S)$ which contains a neighborhood of $K$ in $S$, there exist a $C^1$-neighborhood $\cU$ of $f$
and a neighborhood  $U$ of $K$ such that, for any $g\in\cU$,
\begin{itemize}
\item[--] the maximal invariant set $\La_g$ of  $g$ in $U$ is contained in a submanifold $S_g$, $C^1$-close to $S$,
\item[--] $S_g\cap f(S_g)$ contains a submanifold $S'_g$, $C^1$-close to $S'$,
\item[--] $S_g$ and $S'_g$ depend  continuously on $g$ for the $C^1$-topology.
\end{itemize}
\end{coro}

In order to describe the local dynamics of $f$ and of its perturbations in the neighborhood of $K$
we are therefore reduced to understand the dynamics restricted to $S$. 
As $S$ is a $C^1$-submanifold, the induced local diffeomorphism (defined in a neighborhood of $K$ in $S$) cannot be a priori more regular
than $C^1$.
The standard results on normal hyperbolicity (see~\cite{HPS})
ensure anyway some better smoothness on $S$ when $f$ is more regular.
\footnote{We recall that a map is $C^{k,\alpha}$ with $k\in \NN$ and $\alpha\in(0,1]$
if it is $C^k$ and its $k^\text{th}$ derivative is $\alpha$-H\"older with locally uniform H\"older constant.
In particular a $C^{1,1}$-map has a Lipschitz derivative.
For $r\in [0,+\infty)\setminus \NN$, a map is $C^r$ if it is $C^{k,\alpha}$ with $r=k+\alpha$.}

We say that a partially hyperbolic set $K$ is \emph{$r$-normally hyperbolic}, $r\geq 1$, if
there are constants $\lambda_K>1$ and $C>0$ such that
for each $x\in K$, each $n\geq 1$
and each non-zero vectors
$v^s\in E^{ss},v^c\in E^c,v^u\in E^{uu}$ at $x$, one has:
$$C\lambda_K^{n} \; \frac{\|Df^n(x)\cdot v^s\|}{\|v^s\|}<
\left[\frac{\|Df^n(x)\cdot v^c\|}{\|v^c\|}\right]^r<
C^{-1}\lambda_K^{-n} \; \frac{\|Df^n(x)\cdot v^u\|}{\|v^u\|}.$$

\begin{coro}\label{c.smooth}
Under the conclusion of the Main Theorem:
\begin{itemize}
\item[--] if $f$ is $C^r$ and $K$ is $r$-normally hyperbolic, then $S$ can be chosen $C^r$;
\item[--] if $f$ is $C^r$, $r>1$, then $S$ can be chosen $C^{1,\alpha}$ for some $\alpha>0$.
\end{itemize}
\end{coro}

The study of $C^1$-diffeomorphisms uses sometimes approximation by more regular diffeomorphisms (see for instance~\cite{PS1}).
For this reason, when $f$ and $S$ are only $C^1$, we are interested in getting a more regular submanifold, by a $C^1$-perturbation of $f$.

\begin{prop}\label{p.lissage}
Under the conclusion of the Main Theorem,
there exist a neighborhood $U$ of $K$ and a submanifold with boundary $S'\subset S\cap f(S)$ containing
$U\cap S$ with the following property.

There exist a $C^\infty$-diffeomorphism $g$ and some $C^\infty$ submanifolds with boundary $S'_g, S_g$
which are arbitrarily close to $f$, $S'$ and $S$ for the $C^1$-topology such that:
\begin{itemize}
\item[--] the maximal invariant set $\La_g$ of  $g$ in $U$ is contained in $S_g$,
\item[--] $S_g\cap f(S_g)$ contains the submanifold $S'_g$.
\end{itemize}
\end{prop}
Even if $f$ is a smooth diffeomorphisms, we do not know if it is possible to chose $g$  $C^r$-close to $f$, $r>1$,
in Proposition~\ref{p.lissage}.

\subsection{Dynamical consequences}

\paragraph{Partially hyperbolic dynamics with center dimension equal to $1$.}
For a compact partially hyperbolic set without strong stable and strong unstable connections, we may obtain a better description of
the local dynamics if the dimension of the center direction is very small. The following corollary asserts that,
when the center direction is one-dimensional, we can be perturb in order to get  a dynamics which is locally of Morse-Smale type:
\begin{coro}\label{c.central1D}
Let $K$ be a compact invariant set endowed with a partially hyperbolic structure whose center bundle is $1$-dimensional
and assume that $K$ has no strong connection.
Then, there is a compact neighborhood $U$ of $K$ and, for any $C^1$-neighborhood $\cU$ of  $f$, there is a diffeomorphism $g\in\cU$
such that the maximal invariant set $\La_g$ in $U$ consists in:
\begin{itemize}
\item[--] finitely many hyperbolic periodic orbits, contained in the interior of $U$,
\item[--] the orbits of finitely many compact segments, each of them contained in the transverse intersection of the stable and
unstable manifolds of periodic points $x,y\in \La_g$.
\end{itemize}    
\end{coro}
For the initial dynamics $f$, any forward orbit in $K$ accumulates on a periodic circle
or on a periodic orbit, or on a Cantor set (a minimal set conjugated to the return map on a family of sections of an
``exceptional minimal set" for a $C^1$-vector field on a compact surface).
\smallskip

This result may be related with an extension by Pujals-Sambarino of Ma\~n\'e's theorem~\cite{mane1D}
about the hyperbolicity of one-dimensional endomorphisms to higher dimensional diffeomorphisms,
see~\cite{PS1} and generalizations~\cite{PS2,CPS}:
\smallskip

\noindent
\emph{Let $K$ be a compact invariant set of a $C^2$ Kupka-Smale diffeomorphism $f$, with a partially hyperbolic splitting
$T_{K}M=E^{ss}\oplus E^c$, where $E^c$ is one dimensional. 
Consider  the maximal invariant set $\La$ of $f$ in a small neighborhood of $K$.
Then the chain-recurrent set of $f$ in $\La$ consists in finitely many normally hyperbolic attracting periodic circles,
and finitely many hyperbolic sets.}

\paragraph{Palis' hyperbolicity conjecture.}
A conjecture by Palis claims that \emph{any diffeomorphism may be $C^1$-approximated by Axiom A diffeomorphisms,
or by diffeomorphisms which present a homoclinic tangency
\emph{(a non-transverse intersection between the stable and unstable manifolds of a
hyperbolic periodic orbit)} or a heterodimensional cycle
\emph{(two hyperbolic periodic orbits with different stable dimensions linked by two heteroclinic orbits)}}.

This conjecture has been solved on surfaces~\cite{PS1} and our Main Theorem allows to generalize in some cases to higher
dimensions\footnote{A point $x$ is called \emph{chain recurrent} if $f$ admits $\varepsilon$-pseudo orbits starting and ending at $x$,
for any $\varepsilon>0$. On the set $\cR(f)$ of chain recurrent points, one defines a equivalence relation as follows:
two points $x,y\in\cR(f)$ are equivalent if there are $\varepsilon$-pseudo obits starting at $x$ and ending at $y$,
and conversely starting at $y$ and ending at $x$, for any $\varepsilon>0$. The \emph{chain recurrence classes}
are the equivalence classes of this relation, inducing a partition of $\cR(f)$ in invariant compact sets.

A \emph{trapping region} of a diffeomorphism is an open set $U$ such that $f(\overline U)$ is contained in $U$.
A \emph{filtrating set} is the intersection of a trapping region $U$ of $f$ with a trapping region $V$ of $f^{-1}$.
One fundamental property of the trapping regions is that a chain recurrence class of $f$ meeting a trapping region is contained in it.}
(see also~\cite[section 2.7]{CP}).

\begin{coro}\label{c.dim2} Let $U$ be a filtrating set of a diffeomorphism $f$ such that the maximal invariant set $\La$ of $f$ in $U$
admits a partially hyperbolic splitting whose center bundle has its dimension equal to $2$. We assume furthermore that $\La$
has no strong connection. Then, in any $C^1$-neighborhood $\cU$ of $f$ there is a diffeomorphism $g\in\cU$ verifying one f the
two following properties:
\begin{itemize}
\item[--] either there is a hyperbolic  periodic saddle $x\in U$ of $g$ whose invariant manifolds present a homoclinic tangency
along an orbit f a point $y\in U$;
\item[--] or $g$ verifies the ``Axiom A+no cycle condition" in $U$: the chain-recurrent set $\cR(g)\cap U$ in $U$ consists in
finitely many hyperbolic transitive sets.
\end{itemize}
\end{coro}
\bigskip

The previous conjecture has motivated studies of diffeomorphisms
``far from homoclinic tangency" or ``far from heterodimensional cycles". 
For instance, Wen has shown that the minimally non-hyperbolic sets of diffeomorphisms $C^1$-far from tangencies and
from heterodimensional cycles are partially hyperbolic with a $1$ or $2$ dimensional center bundle (see \cite{We}
and its global generalization \cite{C}), reducing the conjecture to the partially hyperbolic setting.

The following proposition shows that the hypothesis ``without strong connection" and ``far from heterodimensional cycle" are related.
We remind that that two hyperbolic periodic orbits are \emph{homoclinically related}
if they belong to a same transitive hyperbolic set; the \emph{homoclinic class}
of a hyperbolic periodic point $p$ is the closure of the union of the transitive hyperbolic sets containing $p$.

\begin{prop}\label{p.strongconnection} Consider a diffeomorphism $f$ and a compact invariant set $K$ admitting
a partially hyperbolic structure whose center bundle is one-dimensional. If there exists a dense sequence
of periodic orbits $(O_n)$ in $K$ that are homoclinically related and whose center Lyapunov exponent is positive and converge to zero,
then:
\begin{itemize}
\item[--] either there are diffeomorphisms arbitrarily $C^1$-close to $f$ having a heterodimensional cycle;
\item[--] or for any periodic point $x\in K$ homoclinically related to the $O_n$
one has  $W^{uu}(x)\cap K=\{x\}$.
\end{itemize}
\end{prop}

When the second case is not satisfied, one says that $K$ has a \emph{strong unstable connection at the periodic point} $x$.
Using a connecting lemma one can then by a $C^1$-perturbation create a strong unstable \emph{homoclinic}
intersection at $x$, i.e. an intersection between the strong unstable manifold of the orbit of $x$ and its stable manifold.
By unfolding this intersection, one can create a strong homoclinic intersection associated to other
periodic orbits: some of them have a center exponent close to $0$ and this allows to create a heterodimensional cycle by
another $C^1$-perturbation, which implies the proposition (see~\cite[section 2.3.2]{Pu} and~\cite[section 2.5]{CP}).
\medskip

At the time we obtained Proposition~\ref{p.strongconnection}, it became for us
the main motivation for this work. More precisely one can ask the following question.

\begin{ques}\label{q.strongconnection}
Let $f$ be a diffeomorphism,
$\cU$ be a $C^1$-neighborhood of $f$ and $p$ be a hyperbolic periodic point
such that for any $g\in \cU$ the homoclinic class $H(p,g)$ of the hyperbolic continuation $p_g$
of $p$ admits a partially hyperbolic structure whose center bundle is one-dimensional
and expanded along the orbit of $p$.
Then one of the following cases holds.
\begin{itemize}
\item[--] Either there is $g\in\cU$ such that $H(p,g)$ has a strong unstable connection at a periodic point $x$ homoclinically
related to $p_g$.
\item[--] Or there exists a non-empty open subset $\cV\subset \cU$ such that $H(p,g)$,
for any $g\in \cV$, has no strong unstable connection.
\end{itemize}
\end{ques}

Assuming that the center bundle of $H(p)$ is not uniformly expanded, one can expect to
show that there exist periodic orbits homoclinically related to $p$ whose center exponent is arbitrarily close to $0$.
In the first case of the question~\ref{q.strongconnection}, the proposition~\ref{p.strongconnection}
gives a heterodimensional cycle after a $C^1$-perturbation of $f$.
In the second case, the Main Theorem shows that the dynamics reduces
to a submanifold transverse to the strong unstable bundle: the center direction becomes
an extremal one-dimensional bundle and~\cite{PS2} contradicts the fact that it is not uniformly expanded.

A positive answer to the question~\ref{q.strongconnection}
would thus be an important progress for the Palis' conjecture.
It has been obtained recently in~\cite[theorem 10]{CP} for quasi-attractors
and (with other results, including the present paper) a weak version of the conjecture has been proved:
any diffeomorphism may be $C^1$-approximated by diffeomorphisms
that are essentially hyperbolic or that present a homoclinic tangency or a heterodimensional cycle.
\bigskip

\paragraph{Invariant foliations of surface hyperbolic sets.}
The local stable set of a hyperbolic set $K$ supports a natural invariant lamination whose leaves are the
stable manifolds. It is sometimes useful to extend it as a foliation $\cL^s$ which is \emph{locally invariant}:
there exists an neighborhood $U$ of $K$ such that for any $x$ close to $K$, the connected components
of $f(\cL^s_x)\cap U$ and of $\cL^s_{f(x)}\cap U$ containing $f(x)$ coincide. This has been used for instance
in the original works on the Newhouse phenomenon~\cite{Ne1,Ne2} and in the proof of the structural stability
for hyperbolic surface diffeomorphisms~\cite{demelo}.
The following well-know result becomes a simple consequence of our Main Theorem.
It asserts that a $C^2$ surface diffeomorphism near a hyperbolic set is ``$C^1$-conjugated to a product".

\begin{coro}\label{c.foliation}
Let $f$ be a $C^2$-surface diffeomorphism and $K$ be an invariant compact set which is hyperbolic.
Then, there exists a $C^1$-foliation locally invariant in a neighborhood of $K$ which is tangent to
the stable bundle of $K$.
If moreover $f$ is $C^r$, $r>2$, then the foliation can be chosen $C^{1,\alpha}$, for some $\alpha>0$.
\end{coro}

We do not assume that $K$ is the maximal invariant set in a neighborhood.
For a classical proof, see~\cite[apendix 1]{PT}.

\paragraph{Newhouse phenomenon in dimension larger or equal to $3$.}
Newhouse has shown~\cite{Ne1,Ne2} that among $C^2$-diffeomorphisms of a surface, the existence of a homoclinic
tangency for $f$ generates an open set $\cU$ of diffeomorphisms close exhibiting:
\begin{itemize}
\item[--] \emph{persistent tangencies}: there exists a transitive hyperbolic set whose local stable and local unstable sets
have a non-transverse intersection for any diffeomorphism $g\in \cU$),
\item[--] \emph{generic wild dynamics}: there exists infinitely many sinks or sources for
any diffeomorphism in a dense G$_\delta$ subset of $\cU$.
\end{itemize}
These properties have been generalized to higher dimension by Palis and Viana~\cite{PV} for diffeomorphisms
exhibiting a sectionally dissipative homoclinic tangency, whereas Romero~\cite{R} has obtained the first
property for diffeomorphisms exhibiting an arbitrary homoclinic tangency.
Their proof tries to reduce to the dimension $2$ by either building a locally invariant
$C^2$ surface which support part of the dynamics, or by building an ``intrinsic two-dimensional
differentiable structure".

The second author and Nicolas Gourmelon have noticed~\cite{CG}
that it is possible to recover these results using the two following ingredients:
\begin{itemize}
\item[--] After perturbation the homoclinic tangency satisfies a generic condition
and the Main Theorem can be applied: the dynamics in a neighborhood of the homoclinic tangency
is contained in a $C^{1,\alpha}$-surface.
This allows to reduce to the dimension $2$ as expected, however the smoothness of the
induced dynamics is a priori less than $C^2$.
\item[--] The Newouse phenomenon on surfaces also holds for $\alpha\in (0,1)$
in the space of $C^{1,\alpha}$-diffeomorphisms whose $C^{1,\alpha}$
norm is bounded, endowed with the $C^1$-topology, 
\end{itemize}

\subsection{Strategy of the proof and structure of the paper}
The Main Theorem is obtained in two steps:
\begin{itemize}
\item[--] In section~\ref{s.criterion} we use Whitney's extension theorem in order to
build a submanifold $S$ tangent to the center direction which contains $K$.
\item[--] In section~\ref{s.invariant} we implement a graph transform argument in order to
modify this submanifold and get the local invariance.
\end{itemize}
The section~\ref{s.preliminaries} is devoted to classical preliminary results. The corollaries are proved in section~\ref{s.consequence}.
\medskip

The general strategy of~\cite{CLY} also follows these two steps but there are two important differences
(beyond the fact that we deal with diffeomorphisms).
In the first step we relate the assumptions of Whitney's theorem
to the lack of strong connections. In the second step, we implement in a different way the
graph transform argument, which explains that the admissibility condition does not appear in our work. A key point in our proof is to choose carefully the neighborhood where the
graph transform is defined: it has to be small, and much thiner along the strong directions
(see Proposition~\ref{p.domaine}).
\medskip

\noindent
{\bf Acknowledgements.} The second author is grateful to Enrique Pujals and Genevi\`eve Raugel
for their comments related to this work.
\section{Preliminaries}\label{s.preliminaries}
In this section we recall results about distances to a compact set and
dominated splittings.

\subsection{Smoothing the distance to a compact set}
We will need to consider a smooth function which evaluates the distance to a compact set.

\begin{prop} \label{p.smoothing}
Let $\Si$ be a compact riemannian manifold with boundary.

Then, there exists a constant $C_\Si>0$ such that for any disjoint compact sets $K,L\subset \Si$,
there is a function $\varphi\colon \Si\to [0,1]$ which is as smooth as the manifold $\Si$,
such that $\varphi^{-1}(0)=K$, $\varphi^{-1}(1)=L$
and such that the norm of the differential $D\varphi$ is bounded by $\frac{C_\Si}{d(K,L)}$,
where $d(K,L)=\inf\{d(x,y)\mid  x\in K\ \mbox{and}\  y\in L\}$.
\end{prop}

We first prove the result in $\RR^n$.

\begin{lemm}\label{l.smoothing}
For $n\geq 1$, there exists a constant $\Delta(n)>0$ such that for any disjoint compact subsets $K,L\subset \RR^n$
there is a smooth function $\varphi\colon \RR^n\to [0,1]$ such that $\varphi^{-1}(0)=K$, $\varphi^{-1}(1)=L$ and whose derivative has a norm bounded by $\Delta(n)\cdot d(K,L)$.
\end{lemm}
\begin{proof}
Let us choose $\varepsilon>0$ small and
introduce a smooth function $h\colon [-(1+\varepsilon),1+\varepsilon]^n\to [0,1]$ which coincides with zero on a neighborhood
of the boundary of the cube $[-(1+\varepsilon),1+\varepsilon]^n$ and with $1$ on a neighborhood of $[-1,1]^n$.

For $k\geq 0$, let $\cP_k$ be the dyadic partitions of $\RR^n$ which is the collection of cubes of the form
$$C=[-2^{-(k+1)},2^{-(k+1)}]+2^{-k}V,$$ where $V$ is a vector in $\ZZ^n$.
One also defines the larger cubes
$$C_{\varepsilon}=[-2^{-(k+1)}(1+\varepsilon),2^{-(k+1)}(1+\varepsilon)]+2^{-k}V,$$
$$\widehat C=[-3\times 2^{-(k+1)},3\times 2^{-(k+1)}]+2^{-k}V.$$
For $k>0$, the cube $\widehat C$ is a union of cubes of $\cP_{k}$.

We associate to the cube $C$ the function $h\colon C_\varepsilon\to [0,1]$ defined by
$$h_C\colon x\mapsto a_k \; h\left((x-V)2^{k+1}\right),$$
for some $a_k\in (0,2^{-(k)}]$. Its derivative is bounded by some constant $D$, uniform in $k$.
We choose $a_0=1$ so that for $k=0$ the map $h_C$ is bounded from below by $1$ on $C$.

Let $K\subset\RR^n$ be a compact set and consider the collection $\cC$
of cubes $C$ which satisfy:
\begin{itemize}
\item[--] $C\in \cP_k$ for some $k\geq 0$,
\item[--] the larger cube $\widehat C$ is disjoint from $K$,
\item[--] if $k\neq 0$, there is no cube $C'\in \cP_{k-1}$ containing $C$ such that $\widehat C'$ is disjoint from $K$.
\end{itemize}
Note that the cubes of $\cC$ cover $U=\RR^n\setminus K$ and have disjoint interior.
Moreover two cubes $C,C'\in \cC$ that are adjacent belong to partitions $\cP_k,\cP_{k'}$ with $|k'-k|\leq 1$.
In particular any point $x\in U$ belong to at most $2^n$ cubes $C_\varepsilon$ associated to $C\in \cC$. 

The function $\varphi_K$ defined on $\RR^n\setminus K$ by
$$\varphi_K=\sum_{C\in \cC} h_C$$
is thus positive, smooth and bounded by $2^d$.
It is bounded from below by $1$ outside the $2d$-neighborhood of $K$
(this neighborhood is covered by cubes $C\in \cC\cap \cP_0$).
Its derivative is bounded by $2^nD$.
Note also that if the sequence $a_k$ decreases fast enough to zero as $k\to +\infty$,
then one can extend $\varphi_K$ by $0$ on $K$ and get a smooth function of $\RR^n$.
\medskip

If $K,L\subset \RR^n$ are two disjoint compact subsets of $\RR^n$, one may define
$$\varphi=\frac{\varphi_K}{\varphi_K+\varphi_L}.$$
which is smooth, has values in $[0,1]$ and satisfies moreover $\varphi^{-1}(K)=0$ and $\varphi^{-1}(L)=1$.
Its derivative is bounded by
$$\|D\varphi\|\leq 2\frac{\|D\varphi_K\|+\|D\varphi_L\|}{\varphi_K+\varphi_L}.$$
Let us assume that the distance between $K$ and $L$ is equal to $2n$.
The sum $\varphi_K+\varphi_L$ is thus bounded from below by $1$ everywhere and $\|D\varphi\|$
is smaller than $\frac{2^{n+2}D}{2n}d(K,L)$.

One can reduce to the case the distance from $K$ to $L$ is equal to $2n$ by taking the image
by an homothety. The lemma thus holds for $\Delta(n)=\frac{2^{n+2}D}{2n}$.
\end{proof}

We now prove the manifold case.
\begin{proof}[Proof of Proposition~\ref{p.smoothing}]
Let us consider a finite collection of charts
$\psi_i\colon U_i\to \RR^n$ of $\Si$, $i=1\dots,\ell$,
whose union coincides with $\Si$.
Let us choose a partition of the unity, i.e.
some functions $\theta_i\colon U_i\to [0,1]$ such that
\begin{itemize}
\item[--] $\sum_i\theta_i(x)=1$ at any point $x\in \Si$,
\item[--] for each $i$, the support $Q_i$ of $\theta_i$ is a compact subset of $U_i$.
\end{itemize}

If $K,L$ are two compact subsets of $\Sigma$, one can considers
the function $\widetilde \varphi_i\colon \RR^n\to [0,1]$ associated to the compact sets
$K_i:=\psi_i(K\cap Q_i)$ and $L_i:=\psi_i(L\cap Q_i)$ by Lema~\ref{l.smoothing}.
Note that $\varphi_i:=\theta_i\times(\widetilde\varphi_i\circ \psi_i)$ satisfies:
\begin{itemize}
\item[--] $0<\varphi_i(x)<\theta_i(x)$ for points of the interior of $Q_i\setminus (K\cup L)$,
\item[--] $\varphi_i(x)=0$ for $x\in K$,
\item[--] $\varphi_i(x)=\theta_i$ for $x\in L$,
\item[--] the derivative of $\varphi_i$ is bounded by $C_i/d(\psi_i(K_i,L_i)$
where $C_i$ does not depend on $K,L$.
\end{itemize}
Since any $x\in \Si\setminus (K\cup L)$ belongs to the interior of some $Q_i$,
one deduces that
$\varphi:=\sum_i\varphi_i$
is equal to $0$ on $K$, to $1$ on $L$ and has values inside $(0,1)$ elsewhere.
Its derivative is bounded by
$$\|D\varphi\|\leq \sum_i \frac{C_i}{d(\psi_i(K_i),\psi_i(L_i))}
\leq \sum_i \frac{C_i.{L_i}}{d(K\cap Q_i,L\cap Q_i)}$$
where $L_i$ bounds the Lipschitz constant of $\psi_i^{-1}$.
The function $\varphi$ is as smooth as the charts $\psi_i$
ad the manifold $\Si$.
The proposition thus holds with $C_\Si= \sum_i C_i.{L_i}$.
\end{proof}

\subsection{Cone fields and dominated splitting}
We recall here well-known facts about dominated splitting and cone fields.
This section is used in order to control the smoothness of the center manifold
and can be skipped at a first reading.
\begin{defi}
A \emph{continuous cone field} $\cC$ of dimension $d$
is a family of closed cones $\cC(x)\subset T_xM$ such that:
\begin{itemize}
\item[--] $\cC(x)=\overline{\interior(\cC(x))}$ for the topology on $T_xM$;
\item[--] for each $x\in M$,
there exists a $d$-dimensional subspace contained in $\interior(\cC(x))\cup\{0\}$
and a $(\dim(M)-d)$-dimensional space disjoint from $\cC(x)\setminus \{0\}$;
\item[--] the set of unit vectors of $\cC(x)$
and the set of unit vectors of $T_xM\setminus \interior(\cC(x))$
depend continuously of $x$ for the Hausdorff topology.
\end{itemize}
\end{defi}
\noindent
The collection of cones $T_xM\setminus \interior(\cC(x))$
is a continuous cone field, called the \emph{dual cone field}.
A $d$-dimensional $C^1$-submanifold $S\subset M$ is \emph{tangent} to $\cC$ if $T_xS\subset \cC(x)$ for each $x\in S$.
\smallskip

\noindent
The cone field $\cC$ is \emph{transverse to a submersion $\pi\colon M\to \Sigma_0$}
if for each $x\in \Si_0$ and $z\in \pi^{-1}(z)$,
the tangent space at $z$ of the fiber $T_z\pi^{-1}(x)$ and $\cC(z)\setminus \{0\}$ are disjoint.

\subsubsection{Contracted cone fields}
The notion of contracted cone field is usually defined for diffeomorphisms.
We allow here non-surjective tangent maps, which will be necessary
when we will consider graph transforms.

Let us consider a $C^1$-map $\Psi\colon U\to M$ defined on an open subset $U\subset M$.

\begin{defi}\label{d.contraction}
For $r\geq 1$, the cone field $\cC$ is \emph{$r$-contracted by $\Psi$} if there exists $\lambda>1$ and $n_0\geq 1$ such that for any $n\geq n_0$,
any $x$ in $U\cap \Psi^{-1}(U)\cap\dots\cap \Psi^{-n+1}(U)$ we have:
\begin{itemize}
\item[--] $D\Psi^n(x).\cC(x)\subset \cC(\Psi^n(x))$,
\item[--] $D\Psi^n(x).u$ is non-zero if $u\in \cC(x)\setminus \{0\}$,
\item[--] for any unit vectors $u,v\in T_xM$ such that $u\in \cC(x)$ and $D\Psi^n(x).v\notin \cC(\Psi^n(x))$,
\begin{equation*}\label{e.contract-new}
\min(\|D\Psi^n(x).u\|, \|D\Psi^n(x).u\|^r)>  \lambda^{n} \|D\Psi^n(x).v\|.
\end{equation*}
\end{itemize}
When $r=1$ we simply say that the cone field is contracted.
\end{defi}

\begin{rems}
\begin{enumerate}
\item The second item implies that if a submanifold $S$ is tangent to $\cC$ and invariant by $\Psi$,
then the restriction $\Psi_{|S}$ is a local diffeomorphism.
\item We want that an $r$-contracted cone field is also $r'$-contracted for any $r'\in [1,r]$.
This is the reason why the minimum $\min(\|D\Psi^n(x).u\|, \|D\Psi^n(x).u\|^r)$
appears in the third item. Up to replace $U$ by an open set $U'$ relatively compact in $U$,
a contracted cone field is also $r$-contracted for some $r>1$
(with the same constant $n_0$).
\item If $\Psi$ is a diffeomorphism and if $\cC$ is contracted, the dual cone field
is contracted by $f^{-1}$.
\end{enumerate}
\end{rems}

Let us define for $n\geq 1$ and $z\in\Psi(U)\cap\dots\cap \Psi^{n}(U)$,
the cone $\cC^n(z):=D\Psi^n(\Psi^{-n}(z)).\cC(\Psi^{-n}(z))$
and for $x\in U\cap \dots\cap \Psi^{-n+1}(U)$,
the cone $\cC^{-n}(x):=D\Psi^{-n}(\Psi^{n}(x)).\cC(\Psi^{n}(x))$.
The following lemma justifies that the cone field is contracted.
\begin{lemm}\label{l.cone2}
If the cone field $\cC$ is contracted, there exist $C_1>0$, $\lambda>1$
such that for any $n\geq 1$ and $z\in\Psi(U)\cap\dots\cap \Psi^{n}(U)$,
the cone $\cC^n(z)$ is exponentially thin:
there exists a $d$-dimensional space $F\subset \cC^n(z)$ and,
for any unit vector $u\in \cC^n(z)$, there is $w\in F$ such that
$\|w-u\|\leq C_1\lambda^{-n}$. Similarly, for $x\in U\cap \dots\cap \Psi^{-n+1}(U)$,
the cone $T_xM\setminus \cC^{-n}(x)$ is exponentially thin.
\end{lemm}
\begin{proof} The proof will use the following claim.
\begin{claim}\label{c.angle}
There exist $m_0\geq 1$ and  $\sigma>0$ such that
for any $n\geq m_0$, the angle between the vectors $u$ and $v$
in the third item of Definition~\ref{d.contraction}
is bounded from below by $\sigma$.
The same holds for the angle between $D\Psi^n(x).u$ and $D\Psi^n(x).v$,
if this last vector is not zero.
\end{claim}
\begin{proof}
One chooses $k$ in $[m_0/3,2m_0/3]$.
If $m_0$ is large enough we have $\lambda^{k}>2$.
By invariance of $\cC$ we get that
$D\Psi^{k}(x).v\notin \cC(\Psi^{k}(x))$ so that by the cone contraction
$\|D\Psi^{k}(x).u\|\geq 2 \|D\Psi^{k}(x).v\|$.
Since $\|D\Psi^{k}\|$ is bounded, this implies that the angle between $u$ and $v$ is bounded from below,
proving that the angle between $u$ and $v$ is bounded away from zero when $n\geq m_0$.
A similar argument holds for $D\Psi^n(x).u$ and $D\Psi^n(x).v$.
\end{proof}

Let us prove now the statement of the lemma. We set $x=\Psi^{-n}(z)$.
By the definition of the cone field, there exists a $d$-dimensional space $F_0\subset \cC(x)$
such that $F:=D\Psi^n(x).F_0$ is also $d$-dimensional.
Similarly, there exists a transverse $(\dim(M)-d)$-dimensional space $E\subset T_zM$ which is not contained in $\cC(z)$.
One can thus decompose any unit vector $u\in \cC^n(z)$ as $w+v$ with $w\in F$ and $v\in E$.
By definition $u$ has a preimage $u_0 \in \cC(x)$ and $w$ also. Hence, there exists a preimage $v_0\in T_xM$ of $v=u-w$ by $D\Psi^n(x)$.
The cone contraction gives since $u$ is a unit vector
$$\|v\|\leq \lambda^{-n}\frac{\|v_0\|}{\|u_0\|}.$$
One may assume $v_0\neq 0$.
Since $v_0$ is the preimage by $D\Psi^n(x)$ of a vector $v\notin\cC(z)$
and $u_0-v_0$ belongs to $\cC(x)$, the angle between $u_0-v_0$ and $v_0$ is bounded from below by $\sigma$.
One deduces that $\frac{\|v_0\|}{\|u_0\|}$ is uniformly bounded, hence there exists $C_1>0$
such that $\frac{\|v_0\|}{\|u_0\|}<C_1$. This gives the required estimate.

The argument for $T_xM\setminus \cC^{-n}(z)$ is very similar after
noting that $D\Psi^{-n}(z).E$ is a $(\dim(M)-d)$-linear space contained in
$(T_xM\setminus \cC(x))\cup \{0\}$.
\end{proof}

Let us denote by $m(D\Psi^{n}(x))$ the infimum of the
norms $\|D\Psi^{n}(x).u\|$ over unit vectors $u\in T_xM$.
Here is another consequence of cone contraction.

\begin{lemm}\label{l.minimal-norm}
If the cone $\cC$ is contracted, there exists $C_2>0$ such that
for any $n\geq 1$, any $x\in U\cap \Psi^{-1}(U)\cap\dots\cap \Psi^{-n+1}(U)$
and any unit vector $u\in T_xM$, we have
$$\|D\Psi^n(x).u\|\geq C_2 .m(D\Psi^n(x)_{|T_xM\setminus \cC^{-n}(x)}).$$
\end{lemm}
\begin{proof}
Any unit vector $u\in T_xM$ decomposes as $u=u_1+u_2$
such that $u_1\in \cC(x)$ and $u_2\in T_xM\setminus \cC^{-n}(x)$.
By the Claim~\ref{c.angle}, the angle between $D\Psi^k.u_1$ and $D\Psi^k.u_2$ is uniformly bounded
away from zero for any $k\in \{0,\dots,n\}$
(unless one of these vectors is zero).
As a consequence,
$$\|D\Psi^n.u\|\geq \widetilde C .\max(m(D\Psi^n_{|\cC(x)}).\|u_1\|, m(D\Psi^n_{|T_xM\setminus
\cC^{-n}(x)}).\|u_2\|).$$
This concludes the proof after noting that
$\max(\|u_1\|,\|u_2\|)$ is bounded away from below and that $m(D\Psi^n_{|\cC(x)})\geq m(D\Psi^n_{|T_xM\setminus \cC^{-n}(x)})$ by the cone contraction.
\end{proof}

\subsubsection{Dominated splitting}\label{ss.dominated}
In order to prove higher smoothness in Corollary~\ref{c.smooth},
we extend the usual definition of dominated splitting to the notion of
$r$-dominated splitting. It is related to the dominated splitting as the $r$-hyperbolicity
in~\cite{HPS} is related to the hyperbolicity.

The existence of a dominated splitting and of a contracted cone field are two close properties.
\begin{defi}
Let us consider an invariant compact set $K$ for a diffeomorphism $f$
and an invariant splitting $T_KM=E\oplus F$.
We say that \emph{$E$ is $r$-dominated by $F$}
if there exists $C'>0$, $\lambda>1$ and $n\geq 1$ such that
for any unit vectors $u\in E(x)$ and $v\in F(x)$ we have:
\begin{equation*}
\max(\|Df^n(x).u\|, \|Df^n(x).u\|^r)<  {C'}^{-1}\lambda^{-n} \|Df^n(x).v\|.
\end{equation*}
\end{defi}

One can extend the bundles $E^c,E^{uu}$ as two (non-invariant) continuous bundles $E,F$
over a neighborhood $U$ of $K$.
For any $x\in U$ and any $\beta>0$ one defines the \emph{cone field associated to
the splitting $E\oplus F$ and to the Riemannian metric}: 
\begin{equation}\label{e.dom2}
\cC_\beta(x)=\{w\in T_xM\mid\exists u\in E(x),\;\exists v\in F(x), w=u+v, \|u\|\leq \beta\|v\|\}.
\end{equation}

For $\beta'<\beta$ and $n\geq 1$ such that
$C'\lambda^{-n}\beta<\beta'$, we have for any $x$ close to $K$:
$$Df^n(\cC_\beta(x))\subset \cC_{\beta'}(f(x)).$$

\begin{lemm}\label{l.domination-cone}
If $\Psi$ is a diffeomorphism between $U$ and its image,
if $K\subset U$ is an invariant compact set and $d\geq 1$ an integer,
there exists a dominated splitting $T_KM=E\oplus F$ with $d=\dim(F(x))$ for each $x\in K$,
if and only if there exists a contracted cone field of dimension $d$ on a neighborhood of $K$.

The bundle $F$ is $r$-dominated by $E$ for $f=\Psi^{-1}$ if and only if there exist
a cone field of dimension $d$ on a neighborhood of $K$ which is $r$-contracted by $\Psi$.
\end{lemm}
\begin{proof}
The contracted cone field can be defined from a dominated splitting
as in~\eqref{e.dom2}.
Conversely, if there exists a contracted cone field, $\cC$, we first note that the dual cone
field is contracted by $f^{-1}$. The intersection of the cones $\cC^n(x)$ as in lemma~\ref{l.cone2} defines at each point $x\in K$ a $d$-dimensional space $F(x)$. Considering the dual cones,
we also obtain a $(\dim(M)-d)$-dimensional space $E(x)$ and by the definition of contracted cones,
the splitting $T_xM=E(x)\oplus F(x)$ is dominated. The second part of the lemma is obtained
similarly.
\end{proof}

\subsubsection{Lift to Grasmannian bundles: the $r$-contracted case}
In order to prove that an invariant submanifold is $C^r$, we will prove that
its lift in a Grassmanian bundle is $C^{r-1}$. We  explain here how
to lift the dynamics. The $2$-domination allows to get a domination of the lift dynamics.

Let us fix a contracted continuous cone field $\cC$ of dimension $d$.
Let $p\colon G(d,M)\to M$ be the Grassmannian bundle of $d$-dimensional tangent spaces.
We define $\widehat U$, the interior of the set of $d$-dimensional tangent spaces
$E$ contained in a cone $\cC(x)$ for some $x\in U$.
One gets a surjective submersion $p\colon \widehat U\to U$.
By the second item of the Definition~\ref{d.contraction},
$D\Psi$ induces a continuous map
$$\widehat \Psi\colon \widehat U\to G(d,M).$$
Note that $\widehat \Psi$ is $C^{r-1}$ if $\Psi$ is $C^r$.
Moreover $\widehat \Psi$ is a diffeomorphism if $\Psi$
is a $C^2$-diffeomorphism.

\begin{prop}\label{p.fibre-contract}
If $\Psi$ is $C^2$, the map $\widehat \Psi$ contracts the fibers of $p\colon \widehat U\to U$.
More precisely, there exists $C_3>0$ such that for any $n\geq 1$ and
for $P\in \widehat U\cap \dots\cap \widehat \Psi^{-n+1}(\widehat U)$, denoting $x=p(P)$,
$$\|D\widehat \Psi(P)_{|p^{-1}(x)}\|\leq C_3 \|D\Psi^n(x)_{|T_xM\setminus \cC^{-n}(x)}\|\cdot m(D\Psi^n(x)_{|\cC(x)})^{-1}<C_3\lambda^n.$$
\end{prop}
\begin{proof}
For any two $d$-spaces $P,P'$ in $T_xM$, one can consider the linear map
$L\colon P\to P^\perp$ whose graph is $P'$.
The tangent space at $P\in p^{-1}(x)$ to the fiber of $p$
may thus be identified to the space of linear maps $L\colon P\to P^\perp$
and $D\widehat \Psi$ acts by conjugacy. For $x\in U\cap \Psi^{-1}(U)\cap\dots\cap \Psi^{-n+1}(U)$,
$$D\widehat \Psi^n(P).L=\Pi(P_n,P_n^\perp)\circ D\Psi^n(x)\circ L \circ D\Psi^{-n}(f^n(x))$$
where $\Pi(F,E)$ denotes the orthogonal projection on $E$ parallel to $F$
and $P_n=D\Psi^n(x).P$.
Let us consider a $(\dim(M)-d)$-dimensional space $E'$ in $T_{\Psi^n(x)}M$
disjoint from $\cC(\Psi^n(x))\setminus \{0\}$ and its pre image
$E=D\Psi^{-n}(E')$.
The projection between $E$ and $P^\perp$ parallel to $P$
is uniformly bounded and has a uniformly bounded inverse.
The same holds for the projection between
$D\Psi^n(x).E$ and $(D\Psi^n(x).P)^{\perp}$ parallel to $D\Psi^n(x).P$
(this is a consequence of the claim~\ref{c.angle}).
One deduces that the norm of the linear map of $D\Psi^n(x)$ restricted to $E$
and the norm of $\Pi(P_n,P_n^\perp)\circ D\Psi^n(x)$ restricted to $P^\perp$ 
are equal up to a factor bounded by a uniform constant $C_3$:
$$\|D\widehat \Psi^n(P)\|\leq C_3.\|D\Psi^n(x)_{|E}\|.m(D\Psi^{n}(x)_{|P})^{-1}.$$
Together with the cone contraction, this concludes.
\end{proof}

\begin{prop}\label{p.lift}
If $\Psi$ is $C^2$ and if $\cC$ is $r$-contracted with $r\geq 2$, then
there exists a $(r-1)$-contracted cone field $\widehat \cC$ for $\widehat \Psi$,
of dimension $d$, which is transverse to $p$.
One can build $\widehat \cC$ to contain any
compact set of vectors $v\in TG(d,M)$ such that $Dp.v\in \cC\setminus \{0\}$.
If $\cC$ is transverse to a submersion $\pi$, then one can build $\widehat \cC$
to be transverse to $\pi\circ p$.
\end{prop}
\begin{proof}
We define at each point $P\in G(d,M)$ the space $G(P)$ tangent to the fibers of $p$
and $H(P)$ a transverse space (for instance the normal space to $G(P)$
for an arbitrary Riemannian structure), so that $G, H$ are two smooth transverse bundles
and $Dp$ induces an isomorphism between the bundles $H$ and $TM$.
One can thus pull back the Riemannian metric of $M$
as a metric $\|.\|_H$ on $H$.
Let us consider an arbitrary metric $\|.\|_G$ on the bundle $G$
and define $\|.\|=(\|.\|_H^2+\varepsilon^2 \|.\|_G^2)^{1/2}$
a Riemannian metric on $G(d,M)$ for some $\varepsilon>0$ small.
Note that for any vector $v$ at $P$ we have $\|Dp.v\|\leq \|v\|$
with equality if $v$ is tangent to $H(P)$.

We then define $\widehat \cC(P)$ as the set of vectors $v$ at $P$
such that $\|Dp.v\|\geq \frac 1 {\sqrt 2} \|v\|$ and $Dp.v\in \cC(x)$ with $x=p(P)$.
If $E$ is a $d$-dimensional subspace contained in $\cC(x)$ 
and $F$ a $(\dim(T)-d)$-dimensional transverse subspace disjoint from
$\cC(x)\setminus \{0\}$ then $\widehat \cC(P)$ contains $Dp^{-1}(E)\cap H_P$
and is disjoint from $Dp^{-1}(F)\setminus \{0\}$.
Hence $\widehat \cC$ is a continuous cone field of dimension $d$ transverse to $p$.
If $\cC$ is transverse to a submersion $\pi$,
one may choose for $F$ the tangent space at $x$ of the fiber of $\pi$,
which implies that $\widehat \cC$ is transverse to the submersion $\pi\circ p$.

Let us choose $n_0\geq 1$ large enough.
The small constant $\varepsilon>0$ will be fixed later.
In order to prove that the definition~\ref{d.contraction} is satisfied,
it will be enough to check it for any $n\in \{n_0,\dots,2n_0\}$.
For any $P\in \widehat U\cap\dots\cap \widehat \Psi^{-n+1}(\widehat U)$,
we set $x=p(P)$ and take any $v\in \widehat \cC(P)$.
By invariance of the cone field $\cC$
we have $Dp(\Psi^{n}(P)).(D\widehat \psi^{n}.v)\in \cC(\Psi^{n}(x))$.
One can decompose $v=v^H+v^G$ according to the splitting $H\oplus G$.
By definition of $\widehat \cC$, we have $v^H\in \cC(x)$.
By definition of the metric, we have $\|v^H\|_H\geq \varepsilon \|v^G\|_G$.

The image $w$ of $v^H$ by $D\widehat \Psi^{n}$ decomposes as $w^H+w^G$
with $\|w^G\|_G\leq K\|w^H\|_H$, where $K$ is a constant
which controls the angle between the image $D\widehat \Psi^{n}.E$
of the bundle $E$ and the fibers of $p$ for any $n_0\leq n\leq 2n_0$.
For $\varepsilon$ small we get:
\begin{equation}\label{e.first}
\|w^H\|_H\geq 2\varepsilon \|w^G\|_G.
\end{equation}
Since the metric on $H$ is defined by lifting the metric on $U$,
we get
$$\|w^H\|_H\geq m(D\Psi^{n}(x)_{|\cC(x)}).\|v^H\|_H.$$
By Proposition~\ref{p.fibre-contract} we have
$$C_0 \|D\Psi^{n}(P)_{|T_xM\setminus \cC^{-n}(x)}\|.
m(D\Psi^{n}(x)_{|\cC(x)})^{-1}.\|v^G\|_G\geq \|D\widehat \Psi^{n}(P).v^G\|_G.$$
Hence this gives
$$
m(D\Psi^{n}(x)_{|\cC(x)})^{-2}.\|D\Psi^{n}(P)_{|T_xM\setminus \cC^{-n}(x)}\|.\|w^H\|_H\geq \frac {\varepsilon} {C_0} 
\|D\widehat \Psi^{n}(P).v^G\|_G.$$
Since $\cC$ is $2$-contracted and $n$ is large,
$m(D\Psi^{n}(x)^{-2}_{|\cC(x)}).\|D\Psi^{n}(P)_{|T_xM\setminus \cC^{-n}(x)}\|$
is small and
$$\|w^H\|_H\geq 2{\varepsilon}
\|D\widehat \Psi^{n}(P).v^G\|_G.$$
With~\eqref{e.first}, one deduces
$\|w^H\|_H\geq \frac 1 2 \|D\widehat \Psi^{n}(P).v\|$, so that $D\widehat \Psi^{n}(P).v$
belongs to $\widehat \cC(\Psi^{n}(P))$. This gives the first item
of the definition~\ref{d.contraction}.

By the second item of the definition~\ref{d.contraction}
and since $\|w^H\|_H=\|D\Psi^{n}(x).v^H\|_H$ does not vanish,
the image $D\widehat \Psi^{n}.v$ is non zero. Hence the second
item of the definition~\ref{d.contraction} is satisfied.
\medskip

Let us fix two unit vectors $u\in \widehat \cC(P)$,
$v\in T\widehat U\setminus \widehat \cC^{-n}(P)$
and $n\in \{n_0,\dots,2n_0\}$.
We have
$$\|D\widehat \Psi^{n}(P).u\|\geq
\|Dp(D\widehat \Psi^{n}(P).u)\|=
\|D\Psi^{n}(x).(Dp.u)\|\geq 
\frac 1 2 m(D\Psi^{n}(x)_{|\cC(x)}).$$
For $D\widehat \Psi^n.v$, two cases are possible.
In the first case $\|Dp(D\widehat \Psi^n.v)\|\leq \frac 1 2 \|D\widehat \Psi^n.v\|$. We decompose
$v=v^H+v^G$ and the image $w=D\widehat \Psi^n.v^H$ as $w= w^H+w^G$.
As before we have $\|w^G\|_G\leq K\|w^H\|_H$.
The first case restates as $\|w^H\|_H\leq \varepsilon \|w^G+D\widehat \Psi^n.v^G\|_G$.
Hence if $\varepsilon$ has been chosen small enough
$\|w^G\|_G$ is much smaller than $\|D\widehat \Psi^n.v^G\|_G$.
With Proposition~\ref{p.fibre-contract} one gets
$$\|D\widehat \Psi^n.v\|\leq 3 \varepsilon \|D\widehat \Psi^n_{|p^{-1}(x)}\|_G.\|v^G\|_G
\leq 3 C_3 \|D\Psi^n(x)_{|T_xM\setminus \cC^{-n}(x)}\|\cdot m(D\Psi^n(x)_{|\cC(x)})^{-1}.\varepsilon\|v^G\|_G.$$
Hence by the $r$-contraction of the cone field $\cC$ and since $\varepsilon\|v^G\|_G\leq \|v\|=1$ this gives
$$\frac{\|D\widehat \Psi^n.v\|}{\min(\|D\widehat \Psi^{n}.u\|,
\|D\widehat \Psi^{n}.u\|^{r-1})}\leq
 \frac{6\varepsilon C_3\|D\Psi^n(x)_{|T_xM\setminus \cC^{-n}(x)}\|}
{\min(m(D\Psi^n(x)_{|\cC(x)})^2, m(D\Psi^n(x)_{|\cC(x)})^r)}\leq 6\varepsilon C_3\lambda^{-n}.$$
In the other case, $\|Dp(D\widehat \Psi^n.v))\|\geq \frac 1 2 \|D\widehat \Psi^n.v\|$
and $Dp(D\widehat \Psi^n.v)=D\Psi^n(Dp.v)$ belongs to $TU\setminus \cC$.
By the cone contraction, one gets
$$\|D\widehat \Psi^n.v\|\leq 2\|Dp(D\widehat \Psi^n.v))\|=2\|D\Psi^n(Dp.v)\|\leq
2\lambda^{-n}\min(\|D\Psi ^n(Dp.u)\|,\|D\Psi^n(Dp.u)\|^r).$$
We also have $\|D\Psi ^n(Dp.u)\|=\|Dp(D\widehat \Psi ^n.u)\|\leq
\|D\widehat \Psi ^n.u\|$. Consequently
$$\frac{\|D\widehat \Psi^n.v\|}{\min(\|D\widehat \Psi^{n}.u\|,
\|D\widehat \Psi^{n}.u\|^{r})}\leq 2\lambda^{-n}.$$
In both cases the third item of the Definition~\ref{d.contraction} holds,
hence the cone $\widehat \cC$ is $(r-1)$-contracted.
\medskip

For $\varepsilon>0$ small enough, $\widehat \cC$ contains any compact set of
vectors $v$ satisfying $Dp.v\in \cC\setminus \{0\}$.
\end{proof}

\subsubsection{Lifts to Grasmannian bundles: the bunched case}
In Corollary~\ref{c.foliation} we will prove the existence of locally constant foliations.
These are built from locally constant $C^1$-vector fields that are obtained as invariant section
of the tangent bundle. A different dominated splitting of the lift dynamics is used; it is a consequence
of a bunching property.

\begin{defi}
The cone field $\cC$ is \emph{bunched} if there exists $\lambda>1$ and $n_0\geq 1$ such that for any $n\geq n_0$,
any $x$ in $U\cap \Psi^{-1}(U)\cap\dots\cap \Psi^{-n+1}(U)$, and any unit vectors
$u,v\in \cC(x)$ and $w\in T_xM\setminus \cC^{-n}(x)$, we have:
 $$ \|D\Psi^n(x).w\|< \lambda^{-n} \frac{\|D\Psi^n(x).u\|}
{\|D\Psi^n(x).v\|}.$$
\end{defi}

\begin{rema}\label{r.bunched}
When $\Psi$ is a diffeomorphism and $K$ is a partially hyperbolic invariant set such that
$T_KM=E^{ss}\oplus E^c$ and $\dim(E^c)=1$, then the cone fields $\cC$ associated to $E^c$
as in~\eqref{e.dom2} on a neighborhood of $K$ are bunched.
Indeed by Lemma~\ref{l.minimal-norm}, the vectors in the cone field $TM\setminus \cC^{-n}$
are close to the bundle $E^{ss}$, hence are contracted by forward iterations,
while the vectors in the cone $\cC$ become close to the bundle $E^c$ after few iterations;
since $E^c$ is one-dimensional, the iterates of any two vectors $u,v\in \cC$ are almost
collinear and the ratio $\frac{\|D\Psi^n(x).u\|}{\|D\Psi^n(x).v\|}$ does not decay faster than
the strong stable contraction.
\end{rema}

\begin{prop}\label{p.lift2}
Let $\Psi$ be a $C^2$-diffeomorphism between $U$ and its image,
and $\cC$ be a contracted cone field of dimension $d$.
If the cone field dual to $\cC$ is bunched,
then there exists a continuous cone field $\widehat \cC$
of dimension $\dim(M)$ on $\widehat U\subset G(d,M)$ which is contracted by $\widehat \Psi$
and transverse to the submersion $p$.
\end{prop}
\begin{proof}
With the same notations as in the proof of Proposition~\ref{p.lift},
we define in this case $\widehat \cC(P)$ as the set of vectors $v$ at $P$
such that $\|Dp.v\|\geq \frac 1 2 \|v\|$ and we obtain in this way a continuous
cone field of dimension $\dim(M)$ on $G(d,M)$, transverse to $p$.
Since $\widehat \Psi$ is a local diffeomorphism, $D\widehat \Psi(x).u$
does not vanish on non-zero vectors.
It is enough to prove the cone contraction for any integer
$n\in \{n_0,\dots,2n_0\}$, where $n_0$ is large.

Let us consider $v\in \widehat \cC(P)$. It decomposes as $v=v^H+v^G$.
One can decompose the image $w=D\widehat \Psi ^n.v^H$ as $w= w^H+w^G$.
On the one hand, having chosen $\varepsilon$ small enough, one has
\begin{equation}\label{e.contract1}
2\varepsilon\|w^G\|_G\leq \|w^H\|_H.
\end{equation}
On the other hand, by Lemma~\ref{l.minimal-norm}
$$\|w^H\|_H\geq
m(D\Psi^n).\|v^H\|_H \geq m(D\Psi^n(x)).\varepsilon.\|v^G\|_G
\geq C_2.m(D\Psi^n(x)_{|T_xM\setminus \cC^{-n}(x)}).\varepsilon.\|v^G\|_G.$$
By Proposition~\ref{p.fibre-contract} and the bunching,
we have
\begin{equation*}
\begin{split}
2\varepsilon \|D\widehat \Psi^n.v^G\|_G&\leq
2\varepsilon
C_3 \|D\Psi^n(x)_{|T_xM\setminus \cC^{-n}(x)}\|. m(D\Psi^n(x)_{|\cC(x)})^{-1}.\|v^G\|_G\\
&\leq  2
C_3C_2^{-1}\frac{ \|D\Psi^n(x)_{|T_xM\setminus \cC^{-n}(x)}\|}
{m(D\Psi^n(x)_{|T_xM\setminus \cC^{-n}(x)}).m(D\Psi^n(x)_{|\cC(x)})}\|w^H\|_H\;\leq \|w^H\|_H.
\end{split}
\end{equation*}
Together with~\eqref{e.contract1}, this proves the first item of Definition~\ref{d.contraction}.
\medskip

Let us fix two unit vectors $u\in \widehat \cC(P)$,
$v\in T\widehat U\setminus \widehat \cC^{-n}(P)$
and $n\in \{n_0,\dots,2n_0\}$.
Arguing as in the proof of Proposition~\ref{p.lift},
on the one hand we have
$$\|D\widehat \Psi^{n}(P).u\|\geq
\frac 1 2 m(D\Psi^{n}(x))\geq \frac {C_2}2m(D\Psi^{n}(x)_{|T_xM\setminus \cC^{-n}(x)}) .$$
On the other hand we have
$$\|D\widehat \Psi^n.v\|\leq 3 C_3 \|D\Psi^n(x)_{|T_xM\setminus \cC^{-n}(x)}\|\cdot m(D\Psi^n(x)_{|\cC(x)})^{-1}.\varepsilon\|v^G\|_G.$$
Hence by the bunching of the cone field $\cC$ and since $\varepsilon\|v^G\|_G\leq \|v\|$ this gives
$$\frac{\|D\widehat \Psi^n.v\|}{\|D\widehat \Psi^{n}.u\|}\leq
 \frac{6\varepsilon C_3C_2^{-1} \|D\Psi^n(x)_{|T_xM\setminus \cC^{-n}(x)}\|}
{m(D\Psi^n(x)_{|T_xM\setminus \cC^{-n}(x)}).m(D\Psi^n(x)_{|\cC(x)})}\leq 6\varepsilon C_3C_2^{-1}\lambda^{-n}.$$
This gives the last item of the Definition~\ref{d.contraction},
hence the contraction of the cone field $\cC$.
\end{proof}

\section{Existence of submanifolds carrying a compact set}\label{s.criterion}
Let $K$ be a 
 subset of the $n$-dimensional manifold $M$.

\begin{defi}
At each point $z\in K$ the \emph{tangent set} $T_zK$ of $K$ is defined as follow.

For any chart $\varphi\colon U\to \RR^n$ centered at $z$ and for $\varepsilon>0$,
one considers the compact set
$$
\tau_\varepsilon= \operatorname{Closure}{\left\{v\in\RR^n, \exists x,y\in \varphi(K)\cap B(z,\varepsilon), x\neq y
\quad\mbox{and}\quad v=\frac {x-y}{\|x-y\|}\right\}}.
$$
One denotes by $\tau_0$ the intersection $\bigcap_{\varepsilon>0}\tau_\varepsilon$, and by
$T$ the linear subspace of $\RR^n$ generated  by $\tau_0$. The pull-back $T_zK:=(D_z\varphi)^{-1}(T)$
does not depends on the choice of the chart $\varphi$.
\end{defi}

It is clear that a necessary condition for $K$ to be contained in a $d$-dimensional submanifold of $M$ is
that each $T_xK$ is contained in a continuous subbundle of dimension $d$ of the restriction of $TM$ over $K$.
The next theorem is an easy consequence of Whitney's extension theorem and asserts that this condition is also sufficient.

\begin{theo}\label{t.global} If $K\subset M$ is a compact set and $x\mapsto E(x)\subset T_xM$ is a continuous $d$-dimensional
subbundle defined on $K$, such that $T_xK\subset E(x)$ for any $x\in K$,
then there is a compact $d$-dimensional $C^1$-submanifold with boundary $\Si\subset M$ which contains $K$ in its interior.
Furthermore, $\Si$ is tangent to $E(x)$ at each point $x\in K$.
\end{theo}

We now consider the case $K$ is a partially hyperbolic set.
\begin{coro}\label{c.manifold}
Let $f$ be a $C^1$-diffeomorphism of $M$ and $K$ be a compact invariant set admitting a partially
hyperbolic structure $T_KM=E^c\oplus E^{uu}$ where $d=\dim(E^c)$.
If for each $x\in K$, the intersection $W^{uu}(x)\cap K$ is reduced to $\{x\}$,
then there exists a compact $d$-dimensional $C^1$-submanifold with boundary which contains $K$ in its interior.
Furthermore, at each $x\in K$, it is tangent to $E^c(x)$.
\end{coro}

\begin{rema}\label{r.smooth}
One can assume that $\Sigma\setminus K$ is smooth. Indeed, one can modify $\Sigma$
outside a small neighborhood $U$ of $K$ by an arbitrarily small $C^1$-perturbation, such that
$U\cap \Si$ is smooth. A converging sequence of such perturbation when $U$ decreases to $K$
gives the property.
\end{rema}

\subsection{Whitney's extension theorem: the solution of the local problem}

One can find in \cite[Appendix A]{AR} the following statement.

\begin{whitney} Let $A\subset \RR^d$ be a closed subset and $f\colon A\to \RR^{n-d}$ be a continuous function.
The two following properties are equivalent:
\begin{enumerate}
\item $f$ extends to a $C^1$ function $\Phi\colon\RR^d\to \RR^{n-d}$;
\item there is a continuous map $D$ from $A$ to the space of linear maps $L(\RR^d,\RR^{n-d})$
such that if one defines the function $R\colon A\times A\to \RR^{n-d}$ by 
$$R(x,y)= \left(f(y)-f(x)\right)- D(x)(y-x),$$ 
then for each $z\in A$ the quantity $\frac {\|R(x,y)\|}{\|y-x\|}$ tend to $0$ as $x\neq y$ tends to $z$.
\end{enumerate}
Moreover, if $f$ and $D$ verify the second property, then the extension $\Phi$ can be chosen so that $D_x\Phi=D(x)$
at each $x\in A$.
\end{whitney}

It can be restated as follows.

\begin{coro}\label{c.whitney} Let $K\subset \RR^n$ be a compact set such that
\begin{itemize}
\item[--] every $(n-d)$-dimensional affine space
$(x_1,\dots,x_d)\times \RR^{n-d}$ meets $K$ in at most one point,
\item[--] for every $x\in K$ there is a linear subspace $E(x)\subset \RR^n$ of dimension $d$,
transverse to $\{0\}^d\times\RR^{n-d}$ containing the tangent set $T_xK$ of $K$ at $x$,
\item[--] the map $x\mapsto E(x)$ is continuous.
\end{itemize}
Then $K$ is contained is the graph $\Ga$ of a $C^1$-map $\Phi\colon\RR^d\to\RR^{n-d}$ and
the tangent space $T_x\Ga$ coincides with $E(x)$ at each point $x$ of $K$.
\end{coro}
\begin{proof}
Let us denote by $A$  the projection of $K$ on $\RR^d\times \{0\}^{n-d}$ along the vertical direction:
$K$ is the graph of a function $f\colon A\to \RR^{n-d}$. Since $K$ is compact, this map is continuous.

For each point $x$, the $d$-dimensional space $E(x)$ has been assumed to be transverse to the vertical direction
so that it is the graph of a linear map $D(x)\colon \RR^d\to \RR^{n-d}$. The map $x\mapsto D(x)$ is continuous
since $x\mapsto E(x)$ is continuous.

Consider a point $z\in A$ and $p=(z,f(z))$ in $K$.
The hypothesis that $T_pK\subset E(p)$ means that every $v\in T_pK$ can be writen as $(u,D(z)u)$.  Hence for any $x,y\in A$
in a small neighborhood of $z$, the following quantity is very small:
\begin{equation*}
\frac{\left(x-y,f(x)-f(y)\right)}{\|\left(x-y,f(x)-f(y)\right)\|}-\frac{\left(x-y,D(z)(x-y)\right)}{\|\left(x-y,D(z)(x)-y)\right)\|}.
\end{equation*}

After multiplying by the uniformly bounded quotient $\frac{\left\|(x-y,D(z)(x-y))\right\|}{\|x-y\|}$,
we get that
\begin{equation}
\frac{\left\|(x-y,D(z)(x-y))\right\|}{\|\left(x-y,f(x)-f(y)\right)\|}
\frac{\left(x-y,f(x)-f(y)\right)}{\|x-y\|}-\frac{\left(x-y,D(z)(x-y)\right)}{\|x-y\|}
\to 0
\;
\mbox{when $x\neq y$ tend to $z$}.
\label{e.tangent2}
\end{equation}
Considering the projection on the horizontal coordinates $\RR^d$,
one deduces that $\frac{\left\|(x-y,D(z)(x-y))\right\|}{\|\left(x-y,f(x)-f(y)\right)\|}$ goes to $1$.
Since $\frac{\left\|(x-y,D(z)(x-y))\right\|}{\|x-y\|}$
is uniformly bounded, we deduce from~\eqref{e.tangent2} that
\begin{equation*}
\frac{\left(x-y,f(x)-f(y)\right)}{\|x-y\|}-\frac{\left(x-y,D(z)(x-y)\right)}{\|x-y\|}
\to 0
\;
\mbox{when $x\neq y$ tend to $z$}.
\end{equation*}
By projecting on the vertical coordinates $\RR^{n-d}$
and by continuity of $x\mapsto D(x)$ at $z$, one gets
\begin{equation*}
\frac{f(x)-f(y)-D(x)(x-y)}{\|x-y\|}
\to 0
\;
\mbox{when $x\neq y$ tend to $z$}.
\end{equation*}
This gives the second of the properties of Whitney's extension theorem and this theorem
concludes the proof of the corollary.
\end{proof}
\medskip

Consider now a subset $K$ of the $n$-dimensional manifold $M$ and a $d$-dimensional linear subspace $E(x)\subset T_xM$
at each point $x\in K$. We introduce two definitions:
\begin{itemize}
\item[--] A diffeomorphism $\varphi$ from an open set $U\subset M$ to
$]-1,1[^n$ is called an \emph{adapted chart} of $(K,E)$ if $\varphi(K\cap U)\subset ]-1,1[^d\times\{0\}^{n-d}$
and $D\varphi(x).E(x)$ coincides with the linear space $\RR^d\times\{0\}^{n-d}$ for each $x\in K\cap U$.
\item[--] A pair $(U,\Si)$, where $U\subset M$ is open and $\Si\subset U$ is a sub-manifold,
\emph{carries} $(K,E)$ if $K\cap U\subset \Si$ and $T_x\Si= E(x)$ for each $x\in K\cap U$.
\end{itemize}

\begin{coro}\label{c.adapted} Let us consider a compact set $K\subset M$.
If the map $x\mapsto E(x)$ is continuous on $K$, then
each point $x\in K$ is contained in an adapted chart of $(K,E)$.
\end{coro}
\begin{proof}
Let us choose some coordinates around a point $\in K$ such that the vertical
plane $\{0\}^d\times \RR^{n-d}$ is transverse at $p$ to $E(p)$. As $T_pK\subset E(p)$,
shrinking the chart at $p$ if necessary, one can assume that for $x\neq y$ in $K$ close to $p$, $\frac{x-y}{\|x-y\|}$
does not belong to the vertical $(n-d)$-dimensional plane.
Hence, any vertical $(n-d)$-dimensional affine space in this chart
meets $K$ is at most one point. We can thus apply Corollary~\ref{c.whitney} and gets a $C^1$-graph $\Gamma$.
A chart at $p$ which trivializes the graph is an adapted chart of $(K,E)$.
\end{proof}

\subsection{From local to global}

Theorem~\ref{t.global} is now a consequence of corolary~\ref{c.adapted} and of the following proposition.

\begin{prop}\label{p.global}
Let $K\subset M$ be a compact subset, $d>0$ be an integer and at each $x\in K$ let
$E(x)\subset T_xM$ be a $d$-dimensional subspace such that $K$ is covered by charts adapted to $(K,E)$.

Then there exists an open $d$-dimensional submanifold $\Si\subset M$
such that $K\subset \Si$ and $T_x\Si=E(x)$ for $x\in K$.
\end{prop}
\begin{proof}
Consider a finite covering $\{U_i\}_{i\in\{1,\dots,\ell\}}$ of $K$ by charts adapted to $(K,E)$
and for each $i$ fix some open subset $V_i$ whose closure is contained in $U_i$, such that $K\subset\bigcup_i V_i$.
By induction one will build open sets $W_i$, for $i=1,\dots,\ell$,
containing $O_i:=\overline{\bigcup_{j=1}^i V_j}$ and a submanifold $\Si_i$ such that the pair $(W_i,\Si_i)$ carries $(K,E)$.
The open submanifold $\Si_\ell$ obtained this way will satisfy the conclusion of Proposition~\ref{p.global}
since $(W_\ell,\Si_\ell)$ carries $(K,E)$ and since $K$ is contained in $W_\ell$. 

For the first step of the induction,
one chooses $W_1=U_1$ and $\Si_1$ is the horizontal $d$-dimensional plane in the coordinates of $U_1$.
The other steps are obtained by applying the next Lemma~\ref{l.recurrence} to $U_{i+1},V_{i+1},W_i,\Si_i$ and $O_i$.
\end{proof}

\begin{lemm}\label{l.recurrence}
Let $U$ be an adapted chart of $(K,E)$ and $(W,\Si)$ be a pair carrying $(K,E)$.
Consider an open set $V$ whose closure is contained in $U$ and an open set $O$ whose closure is contained in $W$.
Then there is a pair $(W',\Si')$ carrying $(K,E)$ such that $\overline{V\cup O}\subset W'$.
\end{lemm}

The proof of Lemma~\ref{l.recurrence} is obtained after two intermediate lemmas.
\medskip

In the first lemma we prove that in the coordinates of the adapted chart $U$,
$\Sigma$ can be considered as a graph.
We remind that the adapted chart identifies $U$ with $]-1,1[^n$ and we introduce some constants:
\begin{itemize}
\item[--] $\eta<\delta$ in $]0,1[$ are chosen to define the smaller rectangles $]-\delta,\delta[^n,]-\eta,\eta[^n$ of $U$.
By chosing $\eta$ close to $1$, one can assume that they contain $\overline{V}$,
\item[--] $\varepsilon>0$ allows to define the (open) $\varepsilon$-neighborhood $W_\varepsilon$ of $\overline{O}$ and
the intersection $\Si_\varepsilon:=\Si\cap W_\varepsilon$.
Notice that the pair $(W_\varepsilon, \Si_\varepsilon)$ still carries $(K,E)$.
\end{itemize}

\begin{sublemm}\label{l.graph}
For $\varepsilon>0$ small enough, the intersection $\Si_\varepsilon\cap ]-\delta,\delta[^n$ is the graph of a $C^1$-function
$\Phi\colon S\to ]-\eta,\eta[^{n-d}$ defined on an open subset of $]-\delta,\delta[^d$.
Furthermore, if $z\in S\times]-\delta,\delta[^{n-d}$ belongs to $K$, then $z$ belongs to the graph of $\Phi$.
\end{sublemm}
\begin{proof}
Consider a point $z=(x,0)\in ]-1,1[^d\times \{0\}^{n-d}$. 
\begin{itemize}
\item[--] Assume first that $z\notin K\cap \overline{O}$.
Then the fiber $\{x\}\times [-\delta,\delta]^{n-d}$  is a compact set disjoint from $K\cap\overline{O}$.
For $\varepsilon>0$ small enough, $\Si_\varepsilon$ is contained in an arbitrarily small neighborhood of $K\cap \overline{O}$.
As a consequence there is an open neighborhood $S_z$ of $z$ in $]-1,1[^d\times \{0\}^{n-d}$
and a number $\varepsilon(z)>0$  such that  $(S_z\times [-\delta,\delta]^{n-d})\cap \Si_{\varepsilon(z)}$ is empty. 

\item[--] Assume now that $z\in K\cap \overline{O}$.  The submanifold $\Si$ is tangent at $z$ to
$]-1,1[^d\times \{0\}^{n-d}$; as a consequence, there is an open neighborhood  $V_z$ of $z$, contained in
$(]-1,1[^d\times]-\eta,\eta[^{n-d})\cap W$, such that the intersection of $\Si$ with $V_z$ is a graph over
$V_z\cap (]-1,1[^d\times \{0\}^{n-d})$.
Notice that the difference  $\left(\{x\}\times [-\delta,\delta]^{n-d}\right)\setminus V_z$ is a compact set disjoint from $K$
(as $U$ is an adapted chart). Hence there is  $\varepsilon(z)>0$, and a open neighborhood $V'_z$ of
$\left(\{x\}\times [-\delta,\delta]^{n-d}\right)\setminus V_z$  such that $V'_z$ is disjoint from $\Si_{\varepsilon(z)}$ and $K$.
One chooses an open neighborhood $\tilde S_z$ of $z$ in $]-1,1[^d\times \{0\}^{n-d}$ small enough such that
$\tilde S_z\times[-\delta,\delta]^{n-d}$ is contained in the neighborhood $V_z\cup V'_z$ of $\{x\}\times [-\delta,\delta]^{n-d}$. 

By construction $(\tilde S_z\times[-\delta,\delta]^{n-d})\cap \Si_{\varepsilon(z)}$ is the graph  of a $C^1$-function, defined
over an open neighborhood $S_z$ of $z$ in $]-1,1[^d\times\{0\}^{n-d}$ and with values in $]-\eta,\eta[^{n-d}$.

Moreover if $y\in K\cap (S_z\times[-\delta,\delta]^{n-d})$ then $y$ belongs to $V_z$, hence to $W$.
In particular $y\in \Si$; hence $y$ belongs to the graph of the function above, in particular $y\in\Si_{\varepsilon(z)}$.
\end{itemize}
The constructions above associates to each point $z\in ]-1,1[^d\times \{0\}^{n-d}$ an open neighborhood
$S_z$ (in $]-1,1[^d\times \{0\}^{n-d}$) and a constant $\varepsilon(z)$.
By compactness of $[-\delta,\delta]^d$, one can choose a finite set $X$ such that the open sets $S_z$, $z\in X$,
cover $[-\delta,\delta]^d$. One fixes $\varepsilon>0$ less that the  $\varepsilon(z)$, $z\in X$.
Then $\Si_\varepsilon\cap]-\delta,\delta[^n$
is the graph of a $C^1$-map $\Phi$ from an open subset $S\subset ]-\delta,\delta[^d$ to
$]-\eta,\eta[^{n-d}$.

Finally let $y$ be a point in $K\cap (S\times]-\delta,\delta[^{n-d})$. Then $y$ belongs to some
$S_z\times [-\delta,\delta]^{n-d}$, $z\in X\cap K\cap \overline O$  and we have seen that $y\in \Si_{\varepsilon(z)}$,
hence $y$ belongs to the graph of $\Phi$.
\end{proof}
\medskip

Let $S$ be the open set given by the previous lemma and let $T$ be the open set $S\times]-\delta,\delta[^{n-d}$.
We have proved that $(W_\varepsilon\cup T,\Si_\varepsilon)$ carries $(K,E)$.
We now modify $\Si_\varepsilon$ in the chart $U$ in order to glue it to the horizontal rectangle
$]-\eta,\eta[^d\times \{0\}^{n-d}$.
We thus choose a smooth function $\theta\colon]-1,1[^d\to [0,1]$ which takes the value $0$ in a neighborhood of $[-\eta,\eta]^d$
and takes the value $1$ in a neighborhood of $]-1,1[^d\setminus ]-\delta,\delta[^d$.

\begin{sublemm} \label{l.recollement}
There is an open submanifold $\Si_\varepsilon'$ which coincides with $\Si_\varepsilon$ on $M\setminus]-\delta,\delta[^n$ and
with the graph of $\theta.\Phi\colon S\to ]-\eta,\eta[^{n-d}$ in $]-\delta,\delta[^n$.
Moreover, $(W_\varepsilon\cup T,\Si'_\varepsilon)$ carries $(K,E)$.
\end{sublemm}
\begin{proof}
Let consider the set $\Si'_\varepsilon$ union of $\Si_\varepsilon\setminus]-\delta,\delta[^n $ with the graph of $\theta.\Phi$
(contained in  $S\times]-\eta,\eta[^{n-d}$).
We will cover $M$ by two open sets and verify that the intersection of $\Si'_\varepsilon$
with each of them is a submanifold of dimension $d$:
\begin{itemize}
\item[--] In $M\setminus [-\delta',\delta']^n$, where $\delta'$ is close to $\delta$,
the sets $\Si'_{\varepsilon}$ and $\Si_{\varepsilon}$ coincide.
Indeed by Sublemma~\ref{l.graph} neither $\Si'_{\varepsilon}$, nor
$\Si_{\varepsilon}$ meet $[-\delta,\delta]^d\times (]-\delta,\delta[^{n-d}\setminus ]-\eta,\eta[^{n-d})$.
Moreover taking $\delta'<\delta$ such that
$\varphi$ is equal to $1$ on $]-\delta,\delta[^d\setminus ]-\delta',\delta'[^d$,
one deduces that $\Si'_{\varepsilon}$ and $\Si_{\varepsilon}$ coincide on
$(]-\delta,\delta[^d\setminus [-\delta',\delta']^d)\times ]-\delta,\delta[^{n-d}$.
\item[--] In $]-\delta,\delta[^n$, the set $\Si'_\varepsilon$ is the graph of a $C^1$-function defined on an open set of
$]-\delta,\delta[^d$.
\end{itemize}

This implies that $\Si'_\varepsilon$ is a $d$-dimensional submanifold of $M$, which is
contained in $W_\varepsilon\cup T$ (since $W_\varepsilon$ contains $\Si_\varepsilon$ and $T$ contains the graph of $\theta.\Phi$). 
\medskip

Consider any $z\in K\cap (W_\varepsilon\cup T)$.
If $z\in T$ then Sublemma~\ref{l.graph} implies that $z$ coincides with $(x,\Phi(x))$ for some $x\in S$.
But $\Phi(x)=0$ when $x\in K$ so that $\theta(x).\Phi(x)=\Phi(x)=0$ and $z$ belongs to the graph of $\theta\Phi$
hence to $\Si'_\varepsilon$.
If $z\in W_\varepsilon\setminus T$ then $z$ belongs to $W_\varepsilon\setminus ]-\delta,\delta[^n$
(the points of $K\cap ]-\delta,\delta[^n\cap W_\varepsilon$ belong to $]-\delta,\delta[^n\cap\Si_\varepsilon$
and hence in $T$, by Sublemma~\ref{l.graph}); in particular $x$ belongs to $\Si_\varepsilon\setminus ]-\delta,\delta[^n$ hence to
$\Si'_\varepsilon$. 
\end{proof}
\medskip

\begin{proof}[End of the proof of the Lemma~\ref{l.recurrence}]
By Sublemma~\ref{l.recollement}
the union $\Si'=\Si'_\varepsilon\cup(]-\eta,\eta[^d\times\{0\}^{n-d})$ is an open $C^1$-submanifold of dimension $d$.
Let $W'$ be the open set $W_\varepsilon\cup T\cup ]-\eta,\eta[^n$.
Notice that $\bar O\subset W_\varepsilon\subset W'$ and $\bar V\subset ]-\eta,\eta[^n\subset W'$.

One concludes the proof by showing that $(W',\Si')$ carries $(K,E)$. If $x\in K\cap W'$ then:
\begin{itemize}
\item[--]  Either $x\in ]-\eta,\eta[^n$. Then $x\in ]-\eta,\eta[^d\times\{0\}^{n-d}$ and
since $U$ is a adapted chart of $(K,E)$, the space $E(x)$ is tangent to $]-\eta,\eta[^d\times\{0\}^{n-d}$.
So $x\in\Si'$ and $E(x)$ is tangent to $\Si'$.
\item[--] Or $x\in W_\varepsilon\cup T$. Then $x\in \Si'_\varepsilon\subset \Si'$ and $E(x)$ is tangent to $\Si'_\varepsilon$
(and to $\Si'$) by Sublemma~\ref{l.recollement}.
\end{itemize} 
\end{proof}

The Proposition~\ref{p.global} and the Theorem~\ref{t.global} are now proved.

\subsection{Application to partially hyperbolic sets without strong connections}\label{ss.manifold}

By Theorem~\ref{t.global}, the Corollary~\ref{c.manifold} is a direct consequence of:

\begin{prop}\label{p.manifold}
Under the assumptions of Corollary~\ref{c.manifold}, $T_xK\subset E^c(x)$ at each point $x\in K$.
\end{prop}

As in section~\ref{ss.dominated}, one considers the cone fields $\cC_\beta$ associated to
a continuous extension of the bundles $E^c, E^{uu}$.
To prove the proposition, it is enough to replace $f$ by an iterate $f^k$, hence one can assume
the following properties for any $x$ close to $K$:
$$Df(\cC_1(x))\subset \cC_{\frac12}(f(x)),$$
$$ 2\|Df(u)\|\leq \|Df(v)\| \text{ and } \|Df(v)\|\geq 3 \text{ for any unitary
$u\in E(x),v\in \cC_1(x)$.}$$

For $\varepsilon>0$ less that the radius of the injectivity of the exponential map associated to the metric on the manifold $M$
(if $M$ is not compact, it suffices to consider a compact neighborhood of $K$) and given any two points $x,y\in M$ with
$d(x,y)<\varepsilon$, there is a unique geodesic of length less than $\varepsilon$
joining $x$ to $y$. We will denote by $[x,y]$ this geodesic.
If $\varepsilon$ is small and if $V\subset U$ is a small neighborhood of $K$,
for any two points $x,y\in V$ with $d(x,y)<\varepsilon$, one has $[x,y]\subset U$.

\begin{defi}
\emph{The pair $(x,y)$ is in the cone field $\cC_1$} if the tangent vector of the geodesic segment $[x,y]$ at each point $z\in[x,y]$
belongs to the cone $\cC_1$.
\end{defi}

The pairs of points contained in a same strong unstable leaf can be characterized:

\begin{lemm} \label{l.critere} Let $x, y\in K$ and $m\in \NN$
such that for every $n\geq m$ one has
$$
d(f^{-n}(x),d(f^{-n}(y))\leq\varepsilon\quad \mbox{\emph{and the pair $(f^{-n}(x),f^{-n}(y))$ is in the conefield $\cC_1$}}.
$$ 
Then $x$ and $y$ belong to the same strong unstable manifold.
\end{lemm}
\begin{proof}
From the definition of partial hyperbolicity, for each $n\geq 0$ and each unitary $u\in E^c(x)$:
$$\|Df^{-n}.u\|\geq 2^n\; d(f^{-n}(x),f^{-n}(y)).$$
This implies that $y$ belongs to the strong unstable manifold of $x$
(see~\cite[theorem 5.1]{HPS}).
\end{proof}

The property for a pair of points of $K$ to be in the cone field $\cC_1$ is invariant by positive iterates, as long as the distance
between $x$ and $y$ remains small:
\begin{lemm} \label{l.cone} There is $\delta\in ]0,\varepsilon[$ such that  for any pair $(x,y)$ in $\cC_1$
with $x,y\in K$ and $d(x,y)<\delta$:
\begin{itemize}
\item[--] $d(f(x),f(y))<\varepsilon$,
\item[--]  the pair $(f(x),f(y))$ is in the conefield $\cC_1$,
\item[--]  $d(f(x),f(y))\geq 2 d(x,y)$.
\end{itemize}
\end{lemm}
\begin{proof}
The first point follows from uniform continuity and the two others from these properties:
\begin{itemize}
\item[--] If $d(x,y)$ is small, the image $f([x,y])$ is $C^1$-close to the geodesic segment $[f(x),f(y)]$. In particular,
the ratio of their length is arbitrarily close to $1$.
\item[--] For a pair $(x,y)$ in $\cC_1$ the length of $f([x,y])$ is greater than $3$ times $d(x,y)$ (by
\eqref{e.dom2}). 
\item[--] For a pair $(x,y)$ in  $\cC_1$, the segment $f([x,y])$  is tangent at each point to the cone field $\cC_{\frac12}$.
\end{itemize}
\end{proof}
\medskip

\begin{proof}[Proof of Proposition~\ref{p.manifold}] Assume that there is some point $z\in K$
such that $T_zK\neq E^c(z)$: there exist $v \in T_zK\setminus E^c(z)$
and two sequences $(x_n),(y_n)$ in $K$ converging to $z$
such that in any chart at $z$ the vectors $\frac{y_n-x_n}{\|y_n-x_n\|}$
converge to $v$. Note that by replacing $v$ by an an iterate $Df^\ell.v/\|Df^\ell.v\|$
one can assume $v\in \cC_1(z)$ and hence each pair $(x_n,y_n)$ is in the cone field $\cC_1$. 

For  large $n$ the distance $d(x_n,y_n)$ is arbitrarily small.
The Lemma~\ref{l.cone} implies that there is $k_n>0$ such that $d(f^{k_n}(x_n),f^{k_n}(y_n))\in [\frac\delta{\|Df\|},\delta]$,
and such that $d(f^{k}(x_n),f^{k}(y_n))<\frac\delta{\|Df\|}$ for any $k\in\{0,\cdots,k_n-1\}$.
By taking a subsequence, one can assume that the pairs $(f^{k_n}(x_n),f^{k_n}(y_n))$
converge to a pair $(x,y)$ of points of $K$. We will prove now that $y\in W^{uu}(x)$.

For any $i>0$ the pair $(f^{-i}(x),f^{-i}(y))$ is limit of the pairs $(f^{k_n-i}(x_n), f^{k_n-i}(y_n))$. By the choice of $k_n$
the distances $d(f^{k_n-i}(x_n), f^{k_n-i}(y_n)))$ are less than $\delta$ so that $d(f^{-i}(x),f^{-i}(y))<\varepsilon$.
Applying inductively the Lemma~\ref{l.cone} one gets that the pair $d(f^{k_n-i}(x_n), f^{k_n-i}(y_n)))$ is in the cone field $\cC_1$.
Notice that the geodesic segment $[f^{(k_n-i)}(x_n), f^{(k_n-i)}(y_n)]$ converges (in the $C^1$-topology) to the geodesic segment
$[f^{-i}(x),f^{-i}(y)]$. 
As a consequence  the pair $(f^{-i}(x),f^{-i}(y))$ is in the cone $\cC_1$. 

The Lemma~\ref{l.critere} now concludes that $x$ and $y$ belong to the same strong unstable manifold. Notice that $x\neq y$ because $x$
and $y$ are joined by a geodesic segment with length in
$]0,\varepsilon[$. This contradicts the hypotheses on $K$ in the statement of corollary~\ref{c.manifold}.
\end{proof}
\section{Invariant center manifold}\label{s.invariant}
We explain here how to replace the submanifold given by theorem~\ref{t.global}
by an invariant submanifold.

\begin{theo} \label{t.invariant}
Let $f$ be a diffeomorphism of a manifold $M$ and $K$ be an invariant compact set
contained in the interior of a (a priori non-invariant) compact submanifold with boundary
$\Si$. One assumes furthermore that $K$ admits a partially hyperbolic splitting
$TM=E^c\oplus E^{uu}$ such that $E^c(x)=T_x\Si$ at each point $x\in K$.

Then, one can replace $\Si$ by a submanifold $S$ which is locally invariant:
$S\cap f(S)$ contains a neighborhood of $K$ in $S$. 
\end{theo}
\medskip

The proof follows the usual construction of center manifolds for fixed points:
one considers graphs of functions $h\colon \Si\to \RR^{n-d}$
where $\RR^{n-d}$ is a local coordinates transverse to $\Si$.
The action of $f^{-1}$ can be modified outside a neighborhood of $K$ so that
it preserves the space of Lipschitz graphs. The domination between $E^c$ and $E^{uu}$
ensures that this action is a contraction, hence has a fixed point: this is the center manifold.

In our setting the fixed point has been replaced by an invariant compact set $K$
which makes the argument more delicate: the action of $f^{-1}$ in a neighborhood of $K$
can not be approximated by a linear map and the local transverse coordinates are obtained by
the construction of a tubular neighborhood.

\subsection{First constructions}
In this section we build a tubular neighborhood of the submanifold $\Si$
which will allow us to define the space of graphs.

\subsubsection{Bundles $E,F$ around $K$}
Since the splitting $T_KM=E^c\oplus  E^{uu}$ is partially hyperbolic,
there exist $\lambda_K>1$
and $n_K\geq 1$ such that for each $x\in K$ each $n\geq n_K$
and each unit vector $u\in E^c(x), v\in E^{uu}(x)$ one has:
$$
\|Df^n(x)\cdot v\| > \lambda_K^n \|Df^n(x)\cdot u\| \text{ and }
\|Df^n(x)\cdot v\| > \lambda_K^n.$$
From now on we assume $n_K=1$ (this is always possible by changing the initial metric, see~\cite{G}).

We also extend the bundles $E^c,E^{uu}$, defined on $K$ as two (a priori non-invariant) continuous bundles
$E,F$ on a neighborhood of $K$. One can reduce $\Sigma$ and assume that at each $x\in \Sigma$
$E(x),F(x)$ are defined and that $E(x),T_x\Si$ are transverse to $F(x)$.

\subsubsection{Tubular neighborhoods of $\Si$}\label{ss.tubular}
The next proposition provides us with a tubular neighborhood $T$ of an open submanifold $\Si_0\subset \Si$:
\begin{defi}
A \emph{tubular neighborhood} of an open submanifold $\Si_0\subset \Si$
is a smooth surjective submersion $\pi\colon T\to \Si_0$ on an open neighborhood $T$ of $\Si_0$
which induces the identity on $\Si_0$.
\end{defi}

A vector $u\in T_xM$ for $x\in T$ is {\em vertical} if it is tangent to the fibers of $\pi$
(i.e. $D\pi(x)\cdot u=0$) and {\em horizontal} if it is tangent to $E$.
The set $V_x$ of vertical vectors at $x$ is the {\em vertical space}.
Any tangent vector $u\in T_xM$ decomposes as a sum $u_v+u_h$
where $u_v$ is vertical and $u_h$ is horizontal.
For any $\beta>0$ we denote by $\cC^h_\beta(x)$ the \emph{horizontal cone}
associated to the splitting $E\oplus V$:
$$\cC^h_\beta(x)=\{u\in T_xM, \; \beta \|u_h\| \geq \|u_v\|\}.$$

\begin{prop}\label{p.tubulaire}
 For any $\lambda_0\in]1,\lambda_K[$ and any $\eta,\beta,\delta>0$
there is a neighborhood $\Si_0$ of $K$ in $\Si$
and a tubular neighborhood $\pi\colon T\to \Si_0$ with the following properties:
\begin{enumerate}
\item\label{i.tubulaire1} For any vertical vector $u_v$ at $x\in T\cap f^{-1}(T)$:
$$ \|Df(x)\cdot u_v\| \geq \lambda_0 \|u_v\| \ \mbox{and}\ \|D\pi\circ Df(x)\cdot u_v\|\leq \eta \|u_v\|.$$

\item\label{i.tubulaire2} For any horizontal vector $u_h$ at $x\in T$:
$$(1-\delta)\; \|u_h\|\leq \|D\pi(x)\cdot u_h\|\leq (1+\delta)\;\|u_h\|.$$

\item\label{i.tubulaire3} For any $x\in T\cap f(T)$, one has
$$Df^{-1}(x)\cdot \cC^h_{\beta}(x)\subset \cC^h_{\frac{\beta}{\lambda_0}}(f^{-1}(x)).$$
\end{enumerate}
\end{prop}
\begin{proof}
Let $\widetilde F$ be a smooth vector bundle arbitrarily $C^0$-close to $F$,
let $\Si_0$ be a small open neighborhood of $K$ in $\Si$ and $m_0>0$ be a small constant.
We denote by $\cT_{m_0}$ the $m_0$-neighborhood of the zero section of the restriction
$\widetilde F_{|\Sigma_0}$. The exponential map sends $\cT_{m_0}$ diffeomorphically on a set $T\subset M$.
The canonical projection $\cT_{m_0}\subset \widetilde F_{|\Si_0}$ onto $\Si_0$ induces by identification
a projection $\pi\colon T\to \Si_0$. One checks easily that $T$ is a tubular neighborhood of $\Si_0$.

At points $x$ of $K$ the bundle $E^{uu}$ is preserved, expanded by $Df$ by $\lambda_K$
and the cone $\cC^h_\beta(x)$ is mapped into $\cC^h_{\beta/\lambda_K}(f(x))$.
Since the splitting $E\oplus \widetilde F$ is close to the splitting
$E^c\oplus E^{uu}$, the three items of the proposition follow by continuity.
\end{proof}

As a consequence we obtain:

\begin{lemm}\label{l.verticale}
Under the setting of Proposition~\ref{p.tubulaire} and if $\beta<\frac{\lambda_0-2\eta}{6\beta}$,
for any vertical vector $u_v\neq 0$ at $x\in T\cap f^{-1}(T)$, the image $w=Df(x).u_v$ does not intersect
the cone $\cC_{3\beta}^h(f(x))$.
\end{lemm}
\begin{proof}
One first decomposes $w$ in $w_h+w_v$ and denotes by $w_\pi=D\pi(x).w_h=D\pi(x).w$ the projection by $D\pi$.
We have $\|w\|\geq \lambda_0\|u_v\|$ and $\|w_\pi\|\leq \eta \|u_v\|$ by
Proposition~\ref{p.tubulaire}.(\ref{i.tubulaire1}). Moreover $\|w_\pi\|\geq \frac{1}{2}\|w_h\|$
by Proposition~\ref{p.tubulaire}.(\ref{i.tubulaire2}). Hence,
$$\|w_v\|\geq \|w\|-\|w_h\|\geq \lambda_0\|u_v\|-2\|w_\pi\|\geq (\lambda_0-2\eta)\|u_v\|
\text{ and } \|w_h\|\leq 2\eta\|u_v\|.$$
We thus get $\|w_v\|\geq \frac{\lambda_0-2\eta}{2\eta}\|w_h\|> 3\beta \|w_h\|$ by our choice of the constant $\beta$.
\end{proof}

\subsubsection{Contraction of the tubular neighborhood}\label{ss.tub-contract}
Let $d_\pi$ be the distance along the fibers of the tubular neighborhood $\pi\colon T\to \Si_0$,
 given by the induced metric.
For any $m>0$ small, one denotes by $T_m\subset T$ the set of points $z\in T$
such that $d_\pi(z,\pi(z))<m$.

For any $\theta\in [0,1]$ and $z\in T_m$, one defines $\theta.z\in T_m$ as the barycenter of $z$ and $\pi(z)$
along the geodesic segment joining them in the fiber $\pi^{-1}(\pi(z))$. This gives a map $\Theta\colon (\theta,z)\mapsto \theta.z$
on $[0,1]\times T_m$.
For each $\theta\in [0,1]$, we also denote by $\Theta_\theta\colon T_m\to T_m$, the map $z\mapsto \theta.z$.

Equivalently, let $\widetilde F$ be the tangent bundle on $T$ to the fibers of $\pi$,
let $\cT_m$ be the $m$-neighborhood of the zero section of the restriction $\widetilde F_{|\Si_0}$
and let $\widetilde \Theta\colon [0,1]\times \cT_m\to \cT_m$ be the map which sends $(\theta,v)\in[0,1]\times
\widetilde F(x)$ on $\theta.v$ in the vector space $\widetilde F(x)$.
The maps $\Theta$ and $\widetilde \Theta$ are conjugated by the fibered exponential map.

\begin{prop}\label{p.theta}
\begin{enumerate}
\item\label{i.theta1} The map $\Theta$ is $C^1$.
\item\label{i.theta2} The image of the derivative $D_\theta \Theta$ at a point $(\theta,z)\in [0,1]\times T_m$
is contained in the vertical space $V(z)$ of $z$.
It has a norm bounded by the distance $d_\pi(z,\pi(z))$ in the fiber of $z$.
\item\label{i.theta4} For any $\rho>1$, if $m>0$ is small enough
the following property holds:

For any $\theta\in [0,1]$, any $z,z'\in T_m$ such that $\pi(z)=\pi(z')$, we have
$d_\pi(\theta.z,\theta.z')\leq \rho d_\pi(z,z')$.
\item\label{i.theta3} For any constants $\bar \beta>\beta'>0$, if $m>0$ is small enough
the following property holds:

For any $\theta\in [0,1]$ and $z\in T_m$, the differential $D\Theta_\theta(z)$ sends $\cC_{\beta'}^h(z)$ inside
$\cC_{\bar\beta}^h(\theta.z)$.
\end{enumerate}
\end{prop}
\begin{proof}
The proposition may be obtained easily from the map $\widetilde \Theta$:
\begin{enumerate}
\item The map $\widetilde \Theta$ is $C^1$, so is $\Theta$.
\item For any $v\in \cT_m$, the map $\theta\mapsto \widetilde \Theta(\theta,v)$ is linear.
The norm of its derivatives is equal to $\|v\|$. Let $z\in T_m$ be the image of $v$ by the fibered exponential map.
Since the exponential map sends isometrically the line $\RR.v$ on the geodesic $\RR.z$, one deduces that the derivative
$D_\theta\Theta$ at the point $(\theta,z)$ has a norm equal to $\|v\|=d_\pi(z,\pi(z))$.
\item For any $v,v'\in \cT_m$ which belong to the same fiber $\widetilde F(x)$ of $\widetilde F_{|\Si_0}$
and for any $\theta\in [0,1]$, we have $\|\theta.v-\theta.v'\|\leq \|v-v'\|$.
We consider the images $z,z'\in T_m$ of $v,v'$
by the exponential map. At the point $(x,0)\in \cT_m$, the
derivative of the exponential map is an isometry. Hence, for any $\rho>1$,
if $v$ and $v'$ are close enough to $(x,0)$, we have $d_\pi(\theta.z,\theta.z')\leq \rho d_\pi(z,z')$.
\item By construction, for any point $x\in \Si_0$, and any $\theta\in [0,1]$, the differential $D\Theta_\theta$
of the map $\Theta_\theta$ at $x$ maps the cone $\cC_{\bar\beta}^h(x)$ into itself: the horizontal vectors at $x$ are
preserved and the vertical vectors are contracted. The last item of the proposition thus follows by continuity
of the differential of $\Theta$.
\end{enumerate}
\end{proof}

\begin{rema}\label{r.smooth2}
One can assume that the map $\pi$ and $\Theta$ are smooth on
$\pi^{-1}(\Si_0\setminus K)$ and $[0,1]\times \pi^{-1}(\Si_0\setminus K)$.
Indeed, $\Si_0$ can be taken smooth outside $K$ by Remark~\ref{r.smooth}
and the vector field $\widetilde F$ tangent to the fibers of $\pi$ has been chosen smooth also.
\end{rema}

\subsubsection{Lipschitz functions and graphs over $\Si_0$}
Let us consider $\lambda_0,\eta,\beta$ and a tubular neighborhood $\pi\colon T\to \Si_0$ as given by Proposition~\ref{p.tubulaire}.
\begin{defi}
A \emph{function of $T$} is a map $h\colon U\to T$ where $U$ is a subset of $\Si_0$
such that $\pi\circ h(x)=x$ at any $x\in U$. The image $h(U)$ will be called the \emph{graph} of $h$.

$h$ is \emph{$\beta$-Lipschitz} if the tangent space $T_zh(U)$ is contained in the cone $\cC^h_\beta(z)$
for each $z\in U$.
\end{defi}

The distance inside Lipschitz graphs is bounded:
\begin{lemm}\label{l.length}
Let $\Si'\subset \Si_0$ be a neighborhood of $K$ in $\Si_0$ and $h\colon \Si'\to T$ a $\beta$-Lipschitz
function. Then, for any curve $\sigma\subset \Si'$, we have the estimate
$$|h(\sigma)|\leq 2(1+\beta)|\sigma|$$
where $|\sigma|$ and $|h(\sigma)|$ are the lengths of the curves $\sigma$ and $h(\sigma)$.
\end{lemm}
\begin{proof}
At any point $z\in T$, any vector $u\in \cC_\beta^h(z)$ decomposes as $u_h+u_v$ such that
$\|u_v\|\leq \beta\|u_h\|$. By Proposition~\ref{p.tubulaire}.(\ref{i.tubulaire2}),
we also have $\|u_h\|\leq 2\|v\|$ where $v=D\pi(z)\cdot u$ so that $\|u\|\leq 2(1+\beta)\|v\|$.
This gives the lemma.
\end{proof}

The next lemma will show that the distance between Lipschitz graphs is contracted by $f^{-1}$.
\begin{lemm}\label{l.contracte}
Fix a constant $\gamma>(\lambda_0-4\eta(1+\beta))^{-1}$
and a neighborhood $\Si'\subset \Si_0$ of $K$ in $\Si_0$. Then, there exists a
neighborhood $U_\gamma\subset T$ of $K$ which satisfies the following:

For any graph $S$ of a $\beta$-Lipschitz function $h\colon \Si' \to T$ and any points $z_1,z_2, \widetilde z_2$
in $U_\gamma$ such that $z_1,\widetilde z_2\in S$, $\pi(z_2)=\pi(\widetilde z_2)$ and $\pi(f^{-1}(z_1))=\pi(f^{-1}(z_2))$, we have
$$d_\pi(f^{-1}(z_1),f^{-1}(z_2))\leq \gamma \; d_\pi(z_2,\widetilde z_2).$$
\end{lemm}
\begin{proof}
One chooses $\lambda\in ]1,\lambda_0[$ such that $\gamma>(\lambda-4\eta(1+\beta))^{-1}$.
Let $\sigma$ be a geodesic segment joining the points $f^{-1}(z_1)$ and $f^{-1}(z_2)$.
Since $U_\gamma$ is small, the length of $\sigma$ is small.

The length $|f(\sigma)|$ is close to $\|D f(z).u\|.|\sigma|$
where $z$ is a point of $\sigma$ and $u$ a unit vertical vector tangent to $\sigma$.
Proposition~\ref{p.tubulaire}.(\ref{i.tubulaire1}) gives $|f(\sigma)|\geq \lambda_0|\sigma|$.
The distance between $z_1$ and $z_2$ in $M$ is close to the length of $f(\sigma)$, hence:
\begin{equation}\label{e.compare}
d(z_1,z_2)\geq \lambda|\sigma|.
\end{equation}
Similarly, the length of $\sigma'=\pi\circ f(\sigma)$ is close to $\|D\pi\circ D f(z).u\|.|\sigma|$
and Proposition~\ref{p.tubulaire}.(\ref{i.tubulaire1}) gives
\begin{equation}\label{e.releve}
|\sigma'|\leq 2\eta |\sigma|.
\end{equation}
Since $h$ is $\beta$-Lipschitz function, we have by Lemma~\ref{l.length} that
\begin{equation}\label{e.lip}
|h(\sigma')|\leq 2(1+\beta)|\sigma'|.
\end{equation}
Since the segment $h(\sigma')$ joins the points $z_1$ and $\widetilde z_2=h(\pi(z_2))$, one gets
by~(\ref{e.lip}) and~(\ref{e.releve}) the estimate $d(z_1,\widetilde z_2)\leq 4\eta(1+\beta)|\sigma|$.
Writing $d(z_1,z_2)\leq d(z_1,\widetilde z_2)+d(\widetilde z_2,z_2)$, and using~(\ref{e.compare}), one gets
$$\lambda|\sigma|\leq 4\eta(1+\beta)|\sigma| + d(\widetilde z_2,z_2).$$
Since $|\sigma|=d_\pi(f^{-1}(z_1),f^{-1}(z_2))$, this gives as required
$$d_\pi(f^{-1}(z_1),f^{-1}(z_2))\leq (\lambda-4\eta(1+\beta))^{-1}d(\widetilde z_2,z_2)
\leq \gamma \; d_\pi(z_2,\widetilde z_2).$$
\end{proof}

\subsubsection{Choice of the constants and of the tubular neighborhood}\label{ss.constantes}

We explain what are the constants used in the proofs below and
how to choose them:
\begin{enumerate}
\item One chooses the constants required by proposition~\ref{p.tubulaire}:
one first fixes $\lambda_0\in ]1,\lambda_K[$,
then $\eta>0$ small and finally $\beta,\delta\in]0,1/2[$ small enough so that
\begin{equation}\label{e.choixcone}
\beta<\frac{\lambda_0-2\eta}{6\eta},\quad\quad
\lambda_0 -4\eta(1+\beta)>1
\text{ and } \beta+\delta<\frac 1 {10}.
\end{equation}

\item We also consider $\gamma>0$ and then $\rho >1$ such that
\begin{equation}\label{e.contraction}
(\lambda_0-4\eta(1+\beta))^{-1}<\gamma<1 \text{ and } \gamma \rho <1.
\end{equation}
\end{enumerate}
The first part of~(\ref{e.choixcone}) is used in Lemma~\ref{l.verticale}.
The second part guaranties the existence of $\gamma$ satisfying~\eqref{e.choixcone}
as required in  Lemma~\ref{l.contracte}.
The third one is used in the proof of the Proposition~\ref{p.contracte-cone}
for the cone contraction.
The constant $\rho>1$ appears in Proposition~\ref{p.theta} and the condition $\gamma\rho<1$
will ensure the contraction of the graph transformation in the later constructions.

\begin{enumerate}
\setcounter{enumi}{2}
\item In order to obtain contraction of the horizontal cone fields, we choose $\bar \beta\in ]\frac \beta {\lambda_0}, \beta[$.

\item Proposition~\ref{p.smoothing} applied to $\Si$ gives $C_\Si>0$
and $c_f>1$ is any bound on $\|Df\|$ and $\|Df^{-1}\|$.
\end{enumerate}

One can now fix the tubular neighborhood $T$ which satisfies the conclusions of Proposition~\ref{p.tubulaire}
for the constants $\lambda_0,\eta,\beta$ and the conclusions of Proposition~\ref{p.theta} for $\beta'=\beta/\lambda_0$
and $\bar \beta$.

\subsection{A graph transformation}
\subsubsection{Geometry of the graph images}
Before defining the graph transformation, we need to check that
the images by $f^{-1}$ of the Lipschitz graphs above $\Si_0$ remain Lipschitz graphs.
We recall that $c_f>1$ is a constant which bounds $\|Df\|, \|Df^{-1}\|$.
\begin{prop}\label{p.domaine}
There exists a function $\varepsilon\colon ]0,m_1[\to ]0,+\infty[$, with the following properties.
Let $\Si_m, \widehat \Si_m$
denote the open $\varepsilon(m)$ and $2c_f.\varepsilon(m)$-neighborhoods
of $K$ in $\Si_0$. Then:
\begin{enumerate}
\item\label{i.domaine1} $\lim_{m\to 0}\frac m{\varepsilon(m)}=0$.
\item\label{i.domaine2} For any $m\in ]0,m_1[$,
$f^{-1}(T_m)\cap \pi^{-1}(\widehat \Si_m)\subset T_m$.
\item\label{i.domaine3} For any $m\in ]0,m_1[$, the image $f^{-1}(S)$ of the graph $S$ of any $\beta$-Lipschitz function
$h\colon \widehat \Si_{m}\to T_{m}$
contains the graph of a $\frac{\beta}{\lambda_0}$-Lipschitz function $h'\colon \Si_m\to T_m$ over $\Si_m$.
\end{enumerate}
\end{prop}

The proof uses two preliminary lemmas.
\begin{lemm}\label{l.graphe} For each small neighborhood $\Si'$ of $K$ in $\Si_0$,
there exists a positive constant $m=m(\Si')$ verifying the following property:

Consider the graph $S$ of a $\beta$-Lipschitz function $h\colon \Si'\to T_{m(\Si')}$.
Then, $f^{-1}(S)$ is contained in $T$ and is the graph of a function $h'$
over the subset $\pi(f^{-1}(S))$ of $\Si_0$.
\end{lemm}
\begin{proof}
Working with small charts (where the metric $\|~\|$ and the bundles $E$ and $F$ are almost constant) that cover the tubular neighborhood
$T$, one gets the following property:

\emph{There exists $\nu>0$ such that for any $\beta$-Lipschitz function $h\colon U\to T$ over a subset $U\subset \Si_0$ and
for any points $x,y\in h(U)$ such that $d(x,y)\leq \nu$, then
the geodesic segment joining $x$ to $y$ is tangent to the cone $\cC_{2\beta}^h$ at each point.}

Since $f(K)=K$, since $\Si_0$ is tangent at each point $x\in K$ to $E(x)$ and since $E_{|K}$ is invariant by $Df$,
there exists a neighborhood $\Si_1\subset \Si_0$ of $K$ in $\Si_0$ such that $f^{-1}(\Si_1)$ is contained in $T$
and is the graph of a function over a subset of $\Si_0$.

We now consider $\Si'\subset \Si_1$ and prove that it satisfies the lemma with some constant $m>0$.
Let us assume, by contradiction, that there exists a sequence $(m_n)$ going to $0$, a sequence of graphs $(S_n)$
of $\beta$-Lipschitz functions over $\Si'$ such that $S_n\subset T_{m_n}$ and two sequences of points $(x_n)$
and $(y_n)$ such that for each $n$, the points $x_n$ and $y_n$ are distinct,
contained in $S_n$ and verify $\pi(f^{-1}(x_n))=\pi(f^{-1}(y_n))$.

We first prove that the distance $d(x_n,y_n)$ goes to $0$: in the other case, one would obtain, by considering some
subsequences and using the fact that $m_n$ goes to $0$, two distinct points $x,y$ in $\Si'$ whose images by
$\pi\circ f^{-1}$ coincide, contradicting that $f^{-1}(\Si')\subset f^{-1}(\Si_1)$ is a graph.

We consider the geodesic segment $\sigma_n$ joining $f^{-1}(x_n)$ to $f^{-1}(y_n)$ in the fiber of $\pi^{-1}$.
Then by Lemma~\ref{l.verticale} and our choice of $\beta$, the segment
$f(\sigma_n)$, joining $x_n$ to $y_n$, is never tangent to the cones $\cC_{3\beta}^h$.
On the other hand for $n$ large enough, the points $x_n$ and $y_n$ are at distance less than $\nu$
so that the geodesic segment joining $x_n$ to $y_n$ is tangent to the cone $\cC_{2\beta}^h$.
When $n$ tends to $\infty$, the angles between these two segments at $x_n$ go to $0$ leading to a contradiction.
\end{proof}

\begin{lemm}\label{l.domaine0}
There exists a function $\varepsilon_1\colon ]0,+\infty[\to ]0,+\infty[$, with the following properties:
\begin{enumerate}
\item\label{i.domaine01} $\lim_{m\to 0}\frac m{\varepsilon_1(m)}=0$.
\item\label{i.domaine02} $T_m$ contains the image by $f^{-1}$ of the $\varepsilon_1(m)$-neighborhood  of $K$ in $\Si_0$,
for any $m>0$ small.
\end{enumerate}
\end{lemm}
\begin{proof}
By Lemma~\ref{l.graphe}, there is a neighborhood $\Si_1$ of $K$ in $\Si_0$ such that
$f^{-1}(\Si_1)$ is the graph of a $C^1$-function $h$ which is tangent to the bundle $E$ at points of $K$.
Hence, one can cover $K$ by finitely many open sets $U_i$ of $\Si_0$ and charts $\varphi_i\colon U_i\to V_i\times ]0,1[^{\dim(M)-d}$
such that $V_i$ is an open set of $\Si_0$ and $h$ writes in this chart as a $C^1$-map
from $V_i$ to $V_i\times ]0,1[^{\dim(M)-d}$ of the form $x\mapsto (\pi(x),h_i(x))$.
Moreover, $h_i(x)$ and $Dh_i(x)$ are equal to zero at points $x$ of $K$.

For any $\eta>0$, let $\nu_0(\eta)>0$ be the supremum of the norm of the derivatives
of the maps $h_i$ on the $\eta$-neighborhood of $K$ in $\Si_0$. We set $\nu(\eta)=\max(\eta,\nu_0(\eta))$.
The map $\eta\mapsto \nu(\eta)$ is continuous, increasing and goes to $0$ with $\eta$ (since $Dh_i$ vanishes at points of $K$).
For $m$ small enough, one defines $\eta(m)>0$ as the minimum of the numbers $\eta$ such that $m\leq \eta \nu(\eta)$.
Clearly, $\eta(m)$ goes to $0$ with $m$ so that $\frac{m}{\eta(m)}=\nu(\eta(m))$ goes to $0$ with $m$.

By construction, (using the inequality $m\leq \eta\nu(\eta)$
and the facts that $h$ vanishes on $K$ and has a derivative bounded by $\nu(\eta)$ on the $\eta$-neighborhood of $K$),
the $C^0$-norm of $h$ is bounded by $m$ on the $\eta(m)$-neighborhood of $K$ in $\Si_0$.
In other terms, the graph $f^{-1}(S)$ over the $\eta(m)$-neighborhood of $K$ in $\Si_0$ is contained in $T_m$.
The differential of $f^{-1}$ is bounded by $c_f>1$. Hence, the function $\varepsilon_1(m)=\frac{\eta(m)}{c_f}$
satisfies the statement of the lemma.
\end{proof}
\bigskip

\begin{proof}[Proof of Proposition~\ref{p.domaine}]
We set $\varepsilon(m)=\frac{1}{2c_f}\varepsilon_1(m)$ with $\varepsilon_1$ as in Lemma~\ref{l.domaine0}.
\begin{enumerate}
\item The first item of the lemma is an easy consequence of Lemma~\ref{l.domaine0}.(\ref{i.domaine01}).

\item Let $z'_1$ be a point in $f^{-1}(T_m)\cap \pi^{-1}(\widehat \Si_m)$ and
$z'_2\in f^{-1}(\Si_1)$ such that $x'_1:=\pi(z'_1)=\pi(z'_2)$.
We also let $z_1:=f(z'_1)\in T_m$ and $\widetilde z_1\in \Si_0$ such that
$\pi(\widetilde z_1)=\pi(z_1)$. Since $z_1\in T_m$ we have
$d_\pi(z_1,\widetilde z_1)\leq m$. By Lemma~\ref{l.contracte}, we have
\begin{equation}\label{e.hauteur}
d_\pi(z_1',z'_2)\leq \gamma d_\pi(z_1,\widetilde z_1)\leq \gamma m.
\end{equation}
Since $\Si$ and $f^{-1}(\Si)$ are tangent at points of $K$, for $m$
small we have
$$d_\pi(x'_1,z'_2)\leq d(z'_1,K)\leq 2c_f\varepsilon(m)\leq (1-\gamma).m$$
for $m$ small, by Lemma~\ref{l.domaine0}.(\ref{i.domaine01}).
This proves $d_\pi(x'_1,z'_1)\leq m$ and gives the second item.

\item We now prove the following property:
\begin{equation}\label{e.disjoint}
f^{-1}(T_m\cap \pi^{-1}(\partial \widehat \Si_m))\cap \pi^{-1}(\Si_m)=\emptyset.
\end{equation}

We consider a point $z_1$ in $T_m\cap \pi^{-1}(\partial \widehat \Si_m)$ and $z_2=\pi(z_1)$.
Let $z'_1, z'_2$ be their images by $f^{-1}$ and $x'_1,x'_2$ the projections of $z'_1,z'_2$ by $\pi$.
It is enough to show that $x'_1$ and $K$ are at distance larger than $\varepsilon(m)$ in $\Si_0$.
In particular, one can assume that $x'_1\in \widehat \Si_m$.

The distance between $z_2$ and $K$ in $\Si_0$ equals $2c_f\varepsilon(m)$,
hence the distance between $z'_2$ and $K$ is larger than $2\varepsilon(m)$ in $f^{-1}(\Si_0)$
and (since $f^{-1}(\Si)$ and $\Si$ are tangent at points of $K$)
the distance between $x'_2$ and $K$ is larger than $\frac 3 2 \varepsilon(m)$.
We will show that the distance between $x'_1$ and $x'_2$ in $\Si_0$ is smaller than $\frac{\varepsilon(m)}{2}$
and this will give the announced property.

Since $x'_1\in \widehat \Si_m$, by the second item one has $z'_1\in T_m$, hence $d_\pi(x'_1,z'_1)\leq m$.
Since $d_\pi(z_1,z_2)\leq m$, we have $d(z'_1,z'_2)\leq 2c_f m$ where $c_f$ bounds the derivative of $f^{-1}$.
Since $z_2\in \widehat \Si_m$, it belongs to the $2c_f\varepsilon(m)$-neighborhood of $K$ in
$\Si_0$. We have $2c_f\varepsilon(m)\leq \varepsilon_1(m)$, hence
$z'_2\subset T_{m}$ by Lemma~\ref{l.domaine0}.(\ref{i.domaine02}); that is
$d_\pi(z'_2,x'_2)$ is less than $m$. The distance between $x'_1$ and $x'_2$ is thus bounded by $2(1+c_f)m$.\\
Since $\frac{m}{\varepsilon(m)}$ may be taken arbitrarily small, one deduces
that the distance in $\Si_0$ between $x'_1$ and $x'_2$ is smaller than $\frac{\varepsilon(m)}{2}$, as required,
proving~(\ref{e.disjoint}).

Let us come to the proof of the last item.
By Lemma~\ref{l.graphe}, the image by $f^{-1}$ of the graph $S$ of the function $h$ over $\widehat \Si_m$
is the graph $S'$ of a function $h'$ defined over a subset $D$ of $\Si_0$. By Proposition~\ref{p.tubulaire}.(\ref{i.tubulaire3}),
$h'$ is $\frac{\beta}{\lambda_0}$-Lipschitz. By the second item, $S'\cap \pi^{-1}(\Si_m)$ is contained in $T_m$.
By~(\ref{e.disjoint}), $\Si_m$ and the boundary of $D$ does not intersect. Hence, $\Si_m$ is contained in $D$.
This shows that $S'$ contains the graph of a function defined over $\Si_m$.
\end{enumerate}
\end{proof}

\subsubsection{Definition of the graph transformation}\label{ss.transform}
For $m$ small, we denote:
\begin{description}

\item[$U_m$] the open set $T_m\cap (\pi\circ f)^{-1}(T_m)$. For $m$ small
it is an arbitrarily small neighborhood of $K$.

\item[$Lip_{m,\beta}$] the set of $\beta$-Lipschitz functions $h\colon\Si_0\to T_m$,
which vanish outside $\Si_m$, i.e. $h(x)=x$ for $x\in\Si_0\setminus \Si_m$.
It is endowed with the $C^0$-distance: if $h,h'\in Lip_{m,\beta}$ are two Lipschitz functions, we set
$$d(h,h')=\sup_{x\in \Si_m} d_\pi(h(x),h'(x)).$$
By Arzela-Ascoli's theorem, this space is compact.

\item[$Lip_{m,\beta}(K)$] the subset of functions $h\in Lip_{m,\beta}$ vanishing on $K$
(i.e. $\forall x\in K,\; h(x)=x$).

\item[$\varphi_m$] a smooth function $\Si_0\to [0,1]$,
given  by Proposition~\ref{p.smoothing}, equal to $1$ in the $\frac{\varepsilon(m)}2$-neighborhood of $K$
and to $0$ in a neighborhood of $\Si_0\setminus \Si_m$.
Its derivative is bounded by $\frac{2C_\Si}{\varepsilon(m)}$,
where $C_\Si$ is the constant associated to the manifold $\Si$.

\item[$V_m$] is an open neighborhood of $\Si_0\setminus \Si_m$ in $\Si_0$
where $\varphi_m$ vanishes.

\item[$\phi_m(h)$] the function $\Si_0\to T_m$ associated to a function $h\colon\Si_m\to T_m$
as follow: it is equal to $\varphi_m\cdot h$ on $\Si_m$ and to the identity outside $\Si_m$.
With the notations of Section~\ref{ss.tub-contract}, $\phi_m(h)(x)=\Theta(\varphi_m(x),h(x))$
for each point $x\in \Si_m$.

\item[$G_m(h)$] the function $\Si_0\to T_m$ associated to $h\in Lip_{m,\beta}$ as follow:
by Proposition~\ref{p.domaine}.(\ref{i.domaine3}), the image of the graph of $h$ by $f^{-1}$
contains the graph of a Lipschitz function $h'\colon\Si_m\to T_m$. We set $G_m(h)=\phi_m(h')$.
\end{description}
The map $G_m$ is called the \emph{graph transformation}.
\bigskip

One can see the graph $S'$ of $G_m(h)$ as the image of the graph $S$ of $h$
by some $C^1$-map. Indeed, if one defines in $f(T)$ the map:
\begin{equation}\label{e.psi}
\Psi_m\colon z\mapsto \Theta(\varphi_m\circ \pi\circ f^{-1}(z),f^{-1}(z))
\end{equation}
then the graph $S'$ is the union of $\Psi_m(S)\cap \pi^{-1}(\Si_m)$
with $\Si_0\setminus \Si_m$.

By construction, $S'$ and $\Si_0$ coincide on the open set $V_m\subset \Si_0$.
Note that $S'$ is the union of $V_m$ with $\Psi_m(S\cap U_m)$.

\begin{rema}\label{r.smooth3}
From Remarks~\ref{r.smooth} and~\ref{r.smooth2},
one can construct the map $\Psi_m$ as smooth as the diffeomorphism $f$.
\end{rema}

\subsubsection{Invariance of the space of Lipschitz graphs}
The next proposition shows that the image $G_m(h)$ of a Lipschitz graph is also Lipschitz.
\begin{prop}\label{p.transform}
For $m>0$ small, $G_m$ preserves $Lip_{m,\beta}$.
\end{prop}
\medskip

It follows immediately from the next result.

\begin{prop}\label{p.contracte-cone} For $m>0$ small, and $z\in T_m\cap f(T)$,
the image of the cone $\cC_\beta^h(z)$ is contained in the cone $\cC_\beta^h(\Psi_m(z))$.
Moreover $D\Psi_m(z).u_0$ does not vanish for $u_0\in \cC_\beta^h(z)$.
\end{prop}
\begin{proof}
Let us fix $u_0\in \cC^h_\beta(z)$ and denote $u_1=D\Psi(z).u_0$.
We also set $\widetilde z_1=f^{-1}(z)$ and $\widetilde u_1=Df^{-1}(z).u_0$.
By Proposition~\ref{p.tubulaire}.(\ref{i.tubulaire3}),
$\widetilde u_1$ belongs to $\cC_{\frac{\beta}{\lambda_0}}^h(\widetilde z_1)$.
Note that $u_1$ is the image of $\widetilde u_1$ by the tangent map at $\widetilde z_1$ of
$$P_m\colon x\mapsto \Theta(\varphi_m\pi (x),x).$$
We are aimed to compare $u_1$ with $\widetilde u_1$.
\medskip

\begin{claim}\label{c.estimation-u}
If $m$ is small, for any $\widetilde z_1\in f^{-1}(T_m)$ and $\widetilde u_1\in \cC_{\frac{\beta}{\lambda_0}}^h$
we have
\begin{equation}\label{e.estimation-u}
\bigg|\|DP_m.\widetilde u_1\|-\|\widetilde u_1\|\bigg|\leq 9(\beta+\delta).\|\widetilde u_1\|.
\end{equation}
\end{claim}
\begin{proof}
If $\widetilde z_1$ belongs to $\pi^{-1}(\Sigma\setminus \Sigma_m)$,
then $P_m$ coincides locally with the projection $\pi$ on $\Sigma_0$.
In particular,
the image of $DP_m(\widetilde z_1)$ coincides with the tangent space to $\Sigma_0$, hence is contained in $\cC_{\beta}^h(\Psi_m(z))$ and by Proposition~\ref{p.tubulaire}.\eqref{i.tubulaire2}
we have
$$\bigg|\|DP_m.\widetilde u_1\|-\|\widetilde u_1\|\bigg|\leq \delta \|\widetilde u_1\|.$$

Otherwise, $\widetilde z_1$ belongs to $f^{-1}(T_m)\cap \pi^{-1}(\Si_m)$ and it also belongs to $T_m$
by Proposition~\ref{p.domaine}.(\ref{i.domaine2}), so that
\begin{equation}\label{e.distance}
d_\pi(\widetilde z_1, x_1)\leq m.
\end{equation}
We set $\widetilde u_{1,\pi}=D\pi(\widetilde z_1).\widetilde u_1$,
$x_1=\pi(\widetilde z_1)$
and $\theta_1=\varphi_m(x_1)$.
The image $u_1:=DP_m(z).\widetilde u_1$ is equal to $D\Theta(\theta_1,\widetilde z_1).(D\varphi_m.\widetilde u_{1,\pi},\widetilde u_1)$
and decomposes as:
$$u_1=v+w=D_\theta\Theta(\theta_1,\widetilde z_1)\circ D\varphi_m(x_1).\widetilde u_{1,\pi}+
D\Theta_{\theta_1}(\widetilde z_1).\widetilde u_1,$$
where $\Theta_{\theta_1}$ is the map $z\mapsto\Theta(\theta_1,z)$ as before.
 
By Proposition~\ref{p.theta}.(\ref{i.theta2}), the vertical vector $v$
has a norm bounded by $d_\pi(\widetilde z_1,x_1).\|D\varphi_m\|. \|\widetilde u_{1,\pi}\|$.
The first term of this product is bounded by $m$ by (\ref{e.distance}) and the second by $\frac{2C_\Si}{\varepsilon(m)}$.
Hence,
\begin{equation}\label{e.vert}
\|v\|\leq \frac{2C_\Si m}{\varepsilon(m)}\|\widetilde u_{1,\pi}\|.
\end{equation}
The second vector $w=D\Theta_{\theta_1}(\widetilde z_1).\widetilde u_1$ decomposes as the sum $w_h+w_v$ of a horizontal
and a vertical vectors.
By construction $\pi\circ \Theta_{\varphi(x_1)}=\pi$ so that $D\pi(z_1).w_h=D\pi(\widetilde z_1).\widetilde u_1=\widetilde u_{1,\pi}$.
By Proposition~\ref{p.tubulaire}.\eqref{i.tubulaire2}, $\|w_h\|\geq (-\delta)\; \|\widetilde u_{1,\pi}\|$.
One the other hand, using that $\widetilde u_1$ is contained in a cone $\cC^h_{\beta/\lambda_0}$,
and $\bar\beta>\beta/\lambda_0$,
Proposition~\ref{p.theta}.(\ref{i.theta3}) implies that $w$ belongs to a cone $\cC^h_{\bar \beta}$.
Hence, $\|w_v\|\leq \bar \beta\|w_h\|$.\\
Let us recall that $\bar \beta<\beta$.
By these estimates and~(\ref{e.vert}), the vertical component $u_{1,v}=v+w_v$ of $u_1$ is bounded by
$$\|u_{1,v}\|\leq \|v\|+\|w_v\|\leq \left(\frac{4 C_\Si m}{\varepsilon(m)}\ + \bar\beta \right) \|w_h\|.$$
The horizontal component $u_{1,h}$ of $u_1$ is $w_h$.

If one chooses $m$ small enough, $\frac{2 C_\Si}{\varepsilon(m)}m$ may be assumed arbitrarily small
by Proposition~\ref{p.domaine}.(\ref{i.domaine1}),
so that $\frac{2 C_\Si}{\varepsilon(m)}m +\bar \beta$
is less than $\beta$. Hence $u_1$ belong to the cone $\cC^h_{\beta}(z_1)$.
Moreover we have:
$$\bigg|\|u_1\|-\|\widetilde u_1\|\bigg|\;\leq\;
\bigg|\|\widetilde u_1\|-\|D\Theta_{\theta_1}(\widetilde z_1).\widetilde u_1\|\bigg|+\frac{2C_\Si m}{\varepsilon(m)}\|\widetilde u_{1,\pi}\|.$$
For $m$ small,
$z_1:=P_m(\widetilde z_1)$ and $\widetilde z_1$ are arbitrarily close.
Furthermore $\widetilde u_1$ and $D\Theta_{\theta_1}(\widetilde z_1).\widetilde u_1$
have the same projection by $D\pi$ and are tangent to $\cC_\beta^h$.
Recalling that $\beta,\delta\in ]0,1/2[$,
and using Proposition~\ref{p.tubulaire}.\eqref{i.tubulaire2} we also
deduce
$$\bigg|\|\widetilde u_1\|-\|D\Theta_{\theta_1}(\widetilde z_1).\widetilde u_1\|\bigg|\leq 8(\beta+\delta).\|\widetilde u_1\|.$$
The claim is thus proved in al the cases.
\end{proof}

We have proved that $u_1$ belongs to $\cC_\beta^h$
which gives the first part of the Proposition~\ref{p.contracte-cone} (and the
first item of the Definition~\ref{d.contraction}).
Note that if $u_0$ is non zero, the same holds for $\widetilde u_1=Df^{-1}(z).u_0$.
If $\beta+\delta<1/10$ the estimate~\eqref{e.estimation-u} gives
$\|u_1\|\geq \frac 1 {10} \|\widetilde u_1\|$, hence $u_1$ does not vanishes.
In particular we have obtained the second part of the proposition
(and the second item of the Definition~\ref{d.contraction}).
\end{proof}
\bigskip

In order to control the smoothness of the center manifold we will need the following
additional result that can be skipped at a first reading.

\begin{addendum}\label{a.contracte-cone}
Let us assume that $K$ is $r$-normally hyperbolic. One can choose the tubular neighborhood
$\pi\colon T\to \Si_0$ such that for $\beta>0$ and $m>0$ small enough,
the cone $\cC_\beta^h$ is $r$-contracted by the restriction of $\Psi_m$ to $U_m$.
\end{addendum}
\begin{proof}
Note that the two first items of the Definition~\ref{d.contraction} are satisfied.
In order to get the third one, one has to choose the tubular neighborhood $\pi\colon T\to \Si_0$
carefully.

Since $K$ is $r$-normally hyperbolic and
the decomposition $T_KM=E^c\oplus E^{uu}$ is dominated,
there exists $n_K\geq 1$ such that
for $x\in K$ and each unit vectors
$u^c\in E^c(x)$, $v^u\in E^{uu}(x)$ we have
$$\min(\|Df^{-n_K}(x).u^c\|,\|Df^{-n_K}(x).u^c\|^r)>3 \|Df^{-n_K}(x).v^u\|.$$
The same holds for any unit vectors in some thin continuous cone fields
$\cC^c$ $\cC^{uu}$ containing the bundles $E^c$, $E^{uu}$ respectively
and defined on a neighborhood of $K$.

One will choose the cone field $\cC^{uu}$ to be invariant by $Df$.
For instance for $a>0$ small enough,
one may consider at points $x\in K$ the cone $\cC^{uu}(x):=\cC^{uu}_a(x)$
of vectors $w\in T_xM$ such that the norm of the component along $E^c(x)$
is smaller than $a$ times the norm of the component along $E^{uu}(x)$:
then the image of $\cC^{uu}_a(x)$ by $Df^{-1}$ is contained in $\widetilde {\cC^{uu}}(f(x)):=
\cC^{uu}_{a/ \lambda_0}(f(x))$.
The cone fields $\cC^{uu}$ and $\widetilde{\cC^{uu}}$
may be extended continuously to a neighborhood of $K$.

Let us choose $b>0$ small.
The cones $\cC^{uu}$ and $\widetilde{\cC^{uu}}$ beeing chosen
as above with respect to the $r$-normal hyperbolicity and the domination,
we assume that the tangent spaces to the fibers of $F$ are close enough to
the bundle $E^{uu}$.  Moreover $DP_m$ contracts the fibers of $\pi$,
while its restriction to $E^c_{|K}$ is the identity. Consequently
for any points $x\in K$ and any $m$ small, if we have
$DP_m(x).u\in \widetilde {\cC^{uu}}(x)$ then $DP_m(x).u$ and $u$ are
almost collinear, $u$
belongs to $\cC^{uu}(x)$
and $\|DP_m(x).u\|$ is smaller than $(1+b).\|u\|$.
In particular for any $z\in U_m$ and $u\in T_zM$ we have
$$D\Psi_m(z).u\in \cC^{uu}(z) \Rightarrow
\begin{cases}
u\in \cC^{uu}(z),\\
\|D\Psi_m(z).u\|\leq (1+b)\|Df^{-1}(z).u\|\\
\left\|
\frac {D\Psi_m(z).u}{\|D\Psi_m(z).u\|}- \frac {Df^{-1}(z).u}{\|Df^{-1}(z).u\|}
\right\|\leq b.
\end{cases}$$
Arguing in a similar way, the cone field $\widetilde {\cC^{uu}}$ has the same properties.
\medskip

If $\beta,m>0$ are small and $y$ is close to $K$,
any unit vector $u\in \cC_{\beta /{\lambda_0}}^h(y)$ is close to its image
by $DP_m(y)$ by the Claim~\ref{c.estimation-u};
moreover the cones $\cC^h_\beta$ are contained in the cones $\cC^c$.
In particular we have for any $z\in U_m$,
$$u\in \cC^h_\beta(z) \Rightarrow
\begin{cases}
D\Psi_m(z).u\in \cC^{h}_\beta(\Psi_m(z)),\\
\|D\Psi_m(z).u\|\geq (1-b)\|Df^{-1}(z).u\|.\\
\left\|
\frac {D\Psi_m(z).u}{\|D\Psi_m(z).u\|}- \frac {Df^{-1}(z).u}{\|Df^{-1}(z).u\|}
\right\|\leq b.
\end{cases}$$

In particular for any  $z\in U_m\cap\dots\cap \Psi_m^{-n_K+1}(U)$
and any unit vectors $u,v$ with $u\in \cC^h_\beta(z)$ and
$D\Psi_m^{n_K}(z).v\in \cC^{uu}(\Psi_m^{n_K}(z))$ we have
\begin{equation}\label{e.contract}
\min(\|D\Psi_m^{n_K}(z).u\|,\|D\Psi_m^{n_K}(z).u\|^r)>2 \|D\Psi_m^{n_K}(z).v\|.
\end{equation}

Let us now consider $n\geq 1$ large,
$z\in U_m\cap\dots\cap \Psi_m^{-n+1}(U)$
and any unit vectors $u,v$ with $u\in \cC^h_\beta(z)$ and
$D\Psi_m^{n}(z).v\not\in \cC^{h}_\beta(\Psi_m^{n}(z))$.

One can decompose $v$ as a sum $v^c+v^u$ such that
$D\Psi_m^{n}(z).v^u\in \widetilde{\cC^{uu}}$
and $D\Psi_m^{n}(z).v^c\in D\Psi_m^n(\cC^h_{\beta/\lambda_0})$.
Note that since the image of $v$ is not in $\cC_\beta$, there exists $\widetilde C>0$ uniform such that
$$\|D\Psi_m^{n}(z).v^u\|\geq \widetilde C^{-1}.\|D\Psi_m^{n}(z).v^c\|.$$
If $n_0=\ell .n_K$ we have
by~\eqref{e.contract} that
$$\|D\Psi_m^{n-n_0}(z).v^c\|\leq \widetilde C 2^{-\ell} \|D\Psi_m^{n-n_0}(z).v^u\|.$$
Since $D\Psi_m^{n-n_0}(z).v^u$ belongs to $\widetilde{\cC^{uu}}$,
one deduces that
$D\Psi_m^{n-n_0}(z).v$ belongs to $\cC^{uu}$.
One can thus apply~\ref{e.contract} to any iterate
$D\Psi^k_m.u$, $D\Psi^k_m.v$ such that $0\leq k\leq n-n_0-n_K$
and prove for some $C>0$ uniform the required estimate
$$\|D\Psi^n_m(x).v\|\leq C 2^{-n/n_K} \min(\|D\Psi^n_m(x).u\|, \|D\Psi^n_m(x).u\|^r).$$

\end{proof}

\subsection{Fixed point of the graph transformation}

\subsubsection{Existence of the fixed point}\label{ss.fixed}
\begin{prop}\label{p.fixedpoint}
For $m>0$ small enough, $G_m$ has a unique fixed point in $Lip_{m,\beta}$.
\end{prop}

Since $Lip_{m,\beta}$ is compact, the next lemma implies Proposition~\ref{p.fixedpoint}.

\begin{lemm}
For $m$ small, $G_m$ is a contraction of $Lip_{m,\beta}$.
\end{lemm}

\begin{proof}
Let us consider two Lipschitz functions $h_1,h_2\in Lip_{m,\beta}$.
By Proposition~\ref{p.domaine}.(\ref{i.domaine3}), the images by $f^{-1}$ of their graphs
contain the graphs of two Lipschitz functions $h_1',h_2'\colon\Si_m\to T_m$.
Let $x$ be a point in $\Si_m$. One first wants to estimate the length $d_\pi(h_1'(x),h_2'(x))$.

Let us denote by $z_1$ and $z_2$ the images by $f$ of $h'_1(x)$ and $h'_2(x)$.
We introduce their projections $x_1,x_2\in \Si_m$ by $\pi$, so that $z_1=h_1(x_1)$ and $z_2=h_2(x_2)$.
We also consider the point $\widetilde z_2=h_1(x_2)$.
By Lemma~\ref{l.contracte}, we have
$$d_\pi(h'_1(x),h'_2(x))\leq \gamma d(h_1,h_2).$$
By Proposition~\ref{p.theta}.(\ref{i.theta4}), we also have
$$d_\pi(\varphi_m(x).h'_1(x),\varphi_m(x).h'_2(x))\leq \rho d_\pi(h'_1(x),h'_2(x)).$$
Thus, one gets
$$d(G_m(h_1),G_m(h_2))\leq \gamma\rho d(h_1,h_2).$$
We have chosen $\gamma\rho<1$, so this implies the contraction property for $G_m$.
\end{proof}

\subsubsection{$C^1$-smoothness of the fixed graph}\label{ss.regularity}

In order to prove that the graph $S$ of the function $h\in Lip_{m,\beta}$ fixed by $G_m$
is $C^1$, we apply the following proposition to
the tubular neighborhood $\pi\colon T_m\to \Si_0$, the graph $S$, the map $\Psi_m$
and the open sets $U_m, V_m$.

\begin{prop}\label{p.smooth}
Let us consider a $C^1$ submersion $\pi\colon T\to \Si_0$ and a section $S$.
We assume furthermore that
\begin{enumerate}
\item there exists a $C^1$-map $\Psi\colon U\to T$
defined on an open set $U\subset T$,
which preserves $S$: the restriction of $\Psi$ to $S\cap U$
is a homeomorphism to its image and $\Psi(S\cap U)\subset S$,
\item $S$ is $C^1$ on an open set $V$ of $S$ containing $S\setminus \Psi(S\cap U)$,
\item there exists a continuous cone field $\cC$ on $T$ of dimension
$d=\dim(\Si_0)$ that is transverse to the fibration $\pi$,
contracted by $\Psi$ and such that $S$ is tangent to $\cC$.
\end{enumerate}
Then $S$ is a $C^1$-submanifold.
\end{prop}

\begin{coro}
The fixed point $h\colon \Si_0\to T$ of the map $G_m$ is a $C^1$ function.
\end{coro}

\begin{proof}[Proof of Proposition~\ref{p.smooth}]
Let us consider a point $z_0\in S$. There are two cases:
\begin{itemize}
\item[--] Either there exists an integer $n\geq 0$ and a point $z_{-n}\in S\cap V$
such that for each $0\leq k< n$, the point $\Psi^k(z_{-n})$ belongs
to $S\cap U$ and $\Psi^n(z_{-n})=z_0$.
Since $z_{-n}$ belongs to $V$, the graph $S$ is $C^1$ in a neighborhood of $z_{-n}$.
By definition of contracted cone fields, the restriction of $D\Psi^n$ to the tangent bundle
of $S\cap V$ is non-degenerate: the restriction of $\Psi$ to a neighborhood of $z_0$ in $S$ is a diffeomorphism
to its image and by invariance $S$ is $C^1$ in a neighborhood of $z_0$.
In particular $S$ is tangent to a $d$-dimensional space $L_z\subset T_zM$ transverse to $\pi$
at each point $z$ close to $z_0$.
\item[--] Or, there exists an infinite sequence $(z_{-n})_{n\in \NN}$ of points in $S\cap U$
such that $\Psi(z_{-n})=z_{-n+1}$ for each $n\geq 1$.
In this case, using that $\Psi_{|S\cap U}$ is a homeomorphism on its image,
for each $n\geq 0$, the graph $S$ is tangent to the continuous cone field
$\cC^n=D\Psi^n.\cC$ in a neighborhood of $z_0$. This cone field is exponentially thin around a $d$-dimensional subspace of $T_{z_0}M$
by lemma~\ref{l.cone2}.
The intersection of the $\cC^n(z_0)$ is thus a $d$-dimensional space $L_{z_0}$ and $S$ is tangent to $L_{z_0}$ at $z_0$.
\end{itemize}

We now prove that $z\mapsto L_z$ is continuous at any point. It is clear in the first case.
In the second case, it comes from the fact that $L_z$ is tangent to the thin continuous cone field $\cC^n$
in a neighborhood of $z_0$, for arbitrarily large $n$.

We have thus proved that the section $S$ of $\pi$ has continuous tangent spaces transverse to $\pi$:
this is a $C^1$ transverse section, hence a $C^1$-submanifold.
\end{proof}
\bigskip

Under stronger assumptions, we can prove a higher smoothness.
It is not used in the proof of Theorem~\ref{t.invariant} and may be skipped at a first reading.
\begin{addendum}\label{a.smooth}
Under the assumptions of the Proposition~\ref{p.smooth},
let us suppose furthermore that for some $\alpha\in (0,1]$, the map $\Psi$
and the manifolds $S\cap V$ are $C^{1,\alpha}$ and that the cone field $\cC$
is $(1+\alpha)$-contracted. Then $S$ is the graph of a $C^{1,\alpha}$ function.
\end{addendum}
\begin{proof}
Let us consider a point $z\in S$ and a point $z'\in S$ close to $z$.
We have to estimate the difference between the slopes of $T_zS$ and $T_{z'}S$.
\medskip

In the following we define for $x\in S$
the distance $d(u,T_xS)$ between $T_xS$ and a non-zero vector $u\in T_xM$
as the norm of the linear projection of the unit vector $u/\|u\|$ to the
tangent space of the fiber of $\pi$ containing $x$, parallel to $T_xS$.

By the Claim~\ref{c.angle}, the angle between $T_xS$ and the fiber containing $x$ is uniformly bounded
away from zero, hence the linear projection on the tangent space of the fiber at $x$ and parallel to $T_xS$
has a norm bounded by a constant $C_1>0$.
Similarly, for any point $x\in S\cap U\cap \Psi^{-1}(U)\cap\dots\cap \Psi^{-n_0+1}(U)$ and
$v\in T_xM$ such that $D\Psi^{n_0}(x).v\notin \cC(\Psi^{n_0}(x))$, the angle of $v$ with $T_xS$ is bounded
away from zero so that its projection has norm larger than $C_2^{-1}\|v\|$ for another constant $C_2>0$.
\medskip

The $(1+\alpha)$-contraction of the cone field gives an integer $N\geq 1$ satisfying the following.
\begin{lemm}
There exists $N\geq 1$ such that
for any $x\in S\cap U\cap \Psi^{-1}(U)\cap\dots\cap \Psi^{-N+1}(U)$,
and any unit vector $u\in T_xM$ close to $T_xS$ and any unit $w\in \cC(x)$
$$d(D\Psi^N(x).u, T_{\Psi^N(x)}S)
\leq \frac 1 4\min(\|D\Psi_{|S}^{-N}(x)\|^{-\alpha}, 1)d(u, T_{x}S).$$
\end{lemm}
\begin{proof}
The proof is similar to the contraction lemma~\ref{l.cone2}:
we may choose $v$ such that $u+v$ belongs to $T_xS$
and $D\Psi^N(x).v$ is tangent to the fiber of $\pi$ at $\Psi^N(x)$.
One deduces:
$$d(D\Psi^N(x).u, T_{\Psi^N(x)}S)=\frac{\|D\Psi^N(x).v\|}{\|D\Psi^N(x).u\|}.$$

The distance $d(u,T_xS)$ is the norm of the projection of $u$ to the fiber of $x$
parallel to $T_xS$. It is thus equal to $d(v,T_xS)$ and is larger than:
$$d(u, T_xS)\geq C_2^{-1}\|v\|.$$

The $(1+\alpha)$-cone contraction gives, for any unit vector $w\in \cC(x)$,
$$\frac{\|D\Psi^N(x).v\|}{\|v\|}\leq C\lambda^{-N}\min(\|D\Psi^N(x).w\|^{1+\alpha}, \|D\Psi^N(x).w\|).$$
In particular if $w$ is the most contracted unit vector in $\cC(x)$,
$$\frac{\|D\Psi^N(x).v\|}{\|D\Psi^N(x).u\|}
\leq C\lambda^{-N}\min(\|D\Psi^N(x).w\|^\alpha, 1)\|v\|
\leq C\lambda^{-N}\min(\|D\Psi_{|S}^{-N}(x)\|^{-\alpha}, 1)\|v\|.$$
Putting the inequalities together, one gets the announced estimate, provided $CC_2\lambda^{-N}\leq \frac 1 4$.
\end{proof}

Working in charts, one can identify the tangent spaces $T_xM$ and $T_x'M$
at points close.
Since $D\Psi^N$ is $\alpha$-H\"older continuous, there exits
a constant $C_3>0$ such that for points $x,x'$ close and any unit vector $u$,
$$\|D\Psi^N(x).u-D\Psi^N(x').u\|\leq C_3 d(x,x')^\alpha.$$
\medskip

Let us now finish the proof of the addendum.
Let us denote $(z_{-i})_{0\leq i\leq \ell}$ the backward orbit of $z$ by $\Psi^N$ in $S\cap U$:
it is infinite ($\ell=\infty$) or defined for $i$ smaller than some integer $\ell$.
We fix $\sigma>0$ (small and independent of $z,z'$)
so that the point $z'\in S$ has backward iterates $z'_{-i}=\Psi^{-iN}(z')$ by $\Psi^N$
in $S$ whenever the distance $d(z'_{-i},z_{-i})$ is smaller than $\sigma$.
If $\sigma$ has been chosen small enough,
the distance between $z_{-i}$ and $z'_{-i}$ is smaller than
$$d(z_{-i},z'_{-i})\leq \;2^{i}\prod_{j=0}^{i-1} \|D\Psi^{-N}_{|S}(z_{-j})\|\;d(z,z').$$
Let us consider a sequence of unit vectors $u_{-i}\in TS$ at $z'_{-i}$
such that $u_{-i}$ is colinear to $D\Psi^N(z'_{-(i+1)}).u_{-(i+1)}$ for each $i$.
We denote $v_{-i}$ the unit vector colinear to $D\Psi^N(z_{-(i+1)}).u_{-(i+1)}$

We estimate inductively the distance between $u_{-i}$ and $TS$. There exists $C_4>0$ satisfying:
\begin{equation*}
 \begin{split}
d(u_{-i}, T_{z_{-i}}S)&\leq C_1 \| u_{-i}- v_{-i}\|+
d(v_{-i}, T_{z_{-i}}S)\\
&\leq 2C_1 \frac{C_3d(z_{-(i+1)},z'_{-(i+1)})^\alpha}{\|D\Psi^N(z'_{-(i+1)}).u_{-(i+1)}\|}
+\frac 1 4 \min(1, \|D\Psi^{-N}_{|S}(z_{-i})\|^{-\alpha}) d(u_{-(i+1)}, T_{z_{-(i+1)}}S)\\
&\leq C_4d(z_{-(i+1)},z'_{-(i+1)})^\alpha+\frac 1 4 \min(1, \|D\Psi^{-N}_{|S}(z_{-i})\|^{-\alpha}) d(u_{-(i+1)}, T_{z_{-(i+1)}}S).
 \end{split}
\end{equation*}
We thus obtain for any $k\leq \ell$, 
\begin{equation*}
 \begin{split}
&d(u_0, T_{z}S)\leq  \\
&\quad \leq \sum_{i=1}^{k}4^{-i}C_4\prod_{j=0}^{i-1} \|D\Psi^{-N}_{|S}(z_{-j})\|^{-\alpha}d(z_{-i},z'_{-i})^\alpha+
4^{-k}\prod_{j=0}^{k-1} \min(1, \|D\Psi^{-N}(z_{-j})\|^{-\alpha})d(u_{-k}, T_{z_{-k}}S)\\
&\quad \leq C_4 d(z,z')^{\alpha}+
\min\left(2^{-k}\frac{d(z,z')^\alpha}{d(z_{-k},z'_{-k})^{\alpha}}, 4^{-k}\right).
 \end{split}
\end{equation*}
Three cases are possible:
\begin{itemize}
\item[--] The backward orbit of $z$ is infinite (i.e. $\ell=\infty$) and
the distance $d(z_{-k},z'_{-k})$ is smaller than $\sigma$ for any $k$. In this case
$k$ can be taken arbitrarily large.
\item[--] There is $k\leq \ell$ such that
$d(z_{-k},z'_{-k})$ is of the order of $\sigma$: for some constant $C_5>0$ we have:
$$2^{-k}\frac{d(z,z')^\alpha}{d(z_{-k},z'_{-k})^{\alpha}}\leq C_5 d(z,z')^\alpha.$$
\item[--] The distance $d(z_{-k},z'_{-k})$ is smaller than $\sigma$ for any $k\leq \ell$.
Moreover there exist $j\in \{1,\dots,N\}$ and a point
$$\widetilde z\in (S\setminus \Psi(S\cap U))\cap (S\cap U)\cap\dots\cap \Psi^{-j}(S\cap U)$$
such that $\Psi^j(\widetilde z)=z_{-\ell}$.
Since $\sigma>0$ has ben chosen small, there also exists a point
$\widetilde z'$ in a compact neighborhood of $S\setminus \Psi(S\cap U)$
contained in $V$ such that $\Psi^{j}(\widetilde z')=z'_{-\ell}$.
Since $S$ is $C^{1+\alpha}$ on $V$, one deduces that there exists $C_6>0$
uniform such that
$$d(u_{-\ell}, T_{z_{-\ell}}S)\leq C_6 d(z_{-\ell},z'_{-\ell})^{\alpha}\;
 \leq C_6\;2^{\ell}\prod_{j=0}^{\ell-1} \|D\Psi^{-N}_{|S}(z_{-j})\|^\alpha\;d(z,z')^\alpha$$
so that $d(u_0, T_{z}S)\leq C_4 d(z,z')^\alpha + 2^{-\ell}C_6d(z,z')^\alpha$.
\end{itemize}
In any case, we have shown that the distance between $T_{z}S$ and $T_{z'}S$ is smaller
than $C_7\;d(z,z')^\alpha$ for a uniform constant $C_7$, which ends the proof of the addendum.
\end{proof}

\subsection{Conclusion of the proof of Theorem~\ref{t.invariant}}
The graph transform also fixes the space $Lip_{m,\beta}(K)$,
hence the graph $S$ of the function $h\in Lip_{m,\beta}$ fixed by $G_m$
is a $C^1$-submanifold, having the same dimension as $\Si$ and containing $K$ in its interior.
Note that it can be extended as a submanifold with boundary, still denoted by $S$, by taking the union
with $\Si\setminus \Si_0$.
By definition of $G_m$, the submanifolds $S$ and $f(S)$ coincide on a neighborhood of $K$.
Moreover $T_KS$ is an invariant subbundle transverse to $E^{uu}$, hence coincides with $E^c$ at points of $K$.
The proofs of Theorem~\ref{t.invariant} and of the Main Theorem are now complete.
\section{Consequences}\label{s.consequence}

\subsection{Dynamics in a neighborhood: proof of Corollaries~\ref{c.principal} and~\ref{c.saddle}}
Under the setting of the Main Theorem, one considers $0<\delta\ll\varepsilon$ small, and a neighborhood $U$ of $K$.
Provided $U$ is small enough, any point $x$ in the maximal invariant set of $U$ by $f$ has a strong unstable manifold of size
$\varepsilon$ which intersects $S$ at some unique point $\pi(x)$. Moreover the intersection is transverse,
$d(x,\pi(x))<\delta$
and $\pi(x)$ belongs to a small neighborhood of $K$ in $S$.
In particular $f(\pi(x))$ still belongs to the unstable manifold of size $\varepsilon$ of $f(x)$ and to $S$.
One deduces that for any $n\in \ZZ$, one has $$f^n(\pi(x))=\pi(f^n(x)).$$
Taking $n$ arbitrarily large, the distance $d(x,\pi(x))$ is exponentially smaller than $d(f^n(x),\pi(f^n(x)))$
which is bounded by $\delta$. This proves that $x=\pi(x)$. The maximal invariant set of $U$ is thus contained in $S$.
This proves Corollary~\ref{c.principal}.
\medskip

The Corollary~\ref{c.saddle} is obtained by applying the Main Theorem to $f$ and $f^{-1}$ respectively.
The submanifolds $S$ is built as the intersection of two locally invariant submanifolds $S^{cs},S^{cu}$
containing $K$ and tangent to $E^{ss}\oplus E^c$ and $E^c\oplus E^{uu}$ respectively. 

\subsection{Robustness of the submanifold: proof of Corollary~\ref{c.robustness}}
The definition of the graph transform and the results of sections~\ref{ss.transform}, \ref{ss.fixed}
and~\ref{ss.regularity} allow some flexibility: once the space $Lip_{m,\beta}$
and the function $\varphi_m$ have been chosen, the set $K$ is not considered any more
and the graph transform may be modified into a map which is $C^1$-close to the initial transformation.
In particular one can replace $f$ by any diffeomorphism $g$ that is $C^1$-close to $f$.
The map $\Psi_m$ introduced in~\eqref{e.psi} is then be modified as a $C^\infty$-map
$$\Psi_{m,g}\colon z\mapsto\Theta(\varphi_m\circ g^{-1}(z), g^{-1}(z))$$
which defines a new graph transform, producing a new $C^1$-submanifold $S_g$.
By considering a small neighborhood $U$ of $K$ and arguing as for the proof of Corollary~\ref{c.principal},
one gets that the maximal invariant set of $g$ in $U$ is contained in $S_g$.

The graphs $S_g$ and $g(S_g)$ coincide over an open set of $\Si_0$ which is independent from the diffeomorphism $g$
(the points where $\varphi_m$ equals $1$ in Section~\ref{ss.transform}).
This proves that the restriction of the graph $S_g$ to this open set defines a submanifold $S'_g$
which depends continuously on $g$ for the $C^1$-topology, is contained in $S_g\cap g(S_g)$,
and such that $S'_f$ contains a neighborhood of $K$ in $S_f$.

Provided that two diffeomorphisms $g_1$ and $g_2$ are close enough for the $C^1$-topology,
the fixed points in the space $Lip_{m,\beta}$ are close enough. This proves that the graphs $S_{g_1}$ and
$S_{g_2}$ are $C^0$ close. Both are tangent to some thin cone fields obtained by iteration
(see the section~\ref{ss.regularity}). When $g_1,g_2$ are $C^1$-close, these cone fields are close, hence $S_{g_1},S_{g_2}$
are $C^1$-close. This proves that $g\mapsto S_g$ varies continuously for the $C^1$-topology
and this ends the proof of Corollary~\ref{c.robustness}.

\subsection{Higher regularity}

\subsubsection{$C^r$-regularity of the center manifold: proof of Corollary~\ref{c.smooth}}
Let us continue the proof of the Section~\ref{s.invariant} under the assumption that $f$ is $C^r$
and that the partially hyperbolic set $K$ is $r$-normally hyperbolic for some $r>1$.
The argument to prove that $S$ is $C^r$ follows the ideas of the $C^r$-section theorem
in~\cite{HPS} (although we were not able to apply this theorem directly since
the map $\Phi_m$ which is used for the graph transform is not defined on an invariant domain).

Note that by Remarks~\ref{r.smooth2} and~\ref{r.smooth3},
one can assume that the submersion $\pi\colon T\to \Si_0$ is smooth
and that the graph transform $\Psi:=\Psi_m$ introduced in Section~\ref{ss.transform} is $C^r$.
The Proposition~\ref{p.contracte-cone} and the Addendum~\ref{a.contracte-cone}
provide us with a cone field $\cC:=\cC^h_\beta$ on $T$ of dimension $d$
which is transverse to the submersion $\pi$ and $r$-contracted by $\Psi$.
Moreover $S$ is tangent to $\cC$.
Let us consider the domain $U:=U_m$
and the open set $V:=V_m$, then the Proposition~\ref{p.smooth} applies.
If $r=1+\alpha$ with $\alpha\in (0,1)$, the Addendum~\ref{a.smooth} proves that the submanifold $S$ is $C^r$ also. It remains thus to consider the case $r\geq 2$.

We introduce the Grassmannian bundle $p\colon \widehat T \to T$
of $d$-dimensional tangent subspaces of $T$.
Since $r\geq 2$, the Addendum~\ref{a.smooth} proves that $S$ is the graph of a $C^{1,1}$ map.
Consequently, the tangent spaces to $S$ define a Lipschitz graph $\widehat S$
of the fibration $\widehat \pi:=\pi\circ p \colon \widehat T\to \Si_0$.
The preimages $\widehat U:=p^{-1}(U)$ and $\widehat V:=p^{-1}(V)\cap \widehat S$
are open subsets of $\widehat T$ and $\widehat S$ respectively.
The tangent map $D\Psi$ induces a $C^{r-1}$ map $\widehat \Psi\colon \widehat U\to \widehat T$
and the two first properties of the Proposition~\ref{p.smooth} hold.

Since $\widehat S$ is Lipschitz, the angle between the tangent space to $\widehat S$
and the tangent space to th fibers of $p$ is uniformly bounded away from zero.
The unit tangent vectors to $\widehat S$ are thus contained in a compact set of vectors
$v$ satisfying $Dp.v\neq 0$.
One can apply the Proposition~\ref{p.lift} and get a continuous cone field
$\widehat \cC$ on $\widehat T$, of dimension $d$,
which is transverse to the fibration $\widehat \pi$ and
$(r-1)$-contracted by $\widehat \Psi$. Moreover one can require that
$\widehat S$ is tangent to $\widehat \cC$; indeed the collection of unit vectors in the
tangent sets of $\widehat S$ are contained in a compact set of vectors $v$ satisfying
$Dp.v\in \cC\setminus \{0\}$.

The Proposition~\ref{p.smooth} and the Addendum~\ref{a.smooth} now imply that
$\widehat S$ is a $C^{1,\alpha}$-submanifold of $\widehat T$,
hence that $S$ is $C^{2,\alpha}$, where $\alpha=\min(1,r-2)$.
For any integer $k\leq r-1$, one can repeat this argument inductively $k$ times
and conclude that $S$ is $C^{k,\alpha}$ where $\alpha=\min(1,r-k-1)$.
This proves that $S$ is $C^r$ and gives the first item of Corollary~\ref{c.smooth}.
\bigskip

If $f$ is $C^r$ for some $r>1$,
and if we only assume that $K$ is partially hyperbolic,
the continuity of the tangent map and the compactness of the unit bundle
inside the central bundle $E^c$ on $K$,
imply that $K$ is $(1+\alpha)$-normally hyperbolic for some $\alpha>0$ small.
One deduces that $S$ can be chosen $C^{1,\alpha}$, proving the second
item of Corollary~\ref{c.smooth}.

\subsubsection{Smoothing the submanifold: proof of Proposition~\ref{p.lissage}}
Let $S'\subset int(S\cap f(S))$ be a submanifold with boundary which contains
$K$ in its interior. Let $U$ be a small neighborhood of $K$.
Let us consider a $C^\infty$-diffeomorphism $g_0$ and a $C^\infty$-submanifold $S_g$
close to $f$ and $S$ for the $C^1$-topology.
The image $g_0(S_g)$ is $C^1$-close to $S$ and in particular is arbitrarily $C^1$-close to $S_g$
in a neighborhood of $S'$. More precisely, there exists $S'_g\subset S_g$ and
$\widetilde S'_g\subset g_0(S_g)$ which both project on $S'$ by $\pi$.
One can thus consider a diffeomorphism $\tau$ of a neighborhood of $S'$,
which is a translation along each curve $\pi^{-1}(x)$ of the tubular neighborhood $T$
and which maps $\widetilde S'_g$ on $S'_g$.
Since $\pi$ is $C^\infty$, and
since $S'_g$ and $\widetilde S_g$ are $C^\infty$-submanifolds
which are $C^1$-close, one deduces that $\tau$ is a $C^\infty$-diffeomorphism
which is $C^1$-close to the identity. It can be extended as a smooth diffeomorphism of $M$.
The $C^\infty$ diffeomorphism $g:=\tau\circ g_0$ is $C^1$-close to $f$ and by construction
$S'_g\subset S_g\cap g(S_g)$.
Arguing as in the proof of Corollary~\ref{c.principal}, one shows that the maximal invariant set
$\Lambda_g$ of $U$ is contained in $S_g$. This gives the proposition.

\subsection{Consequences when the center dimension equals $1$ or $2$}

\subsubsection{One-dimensional center bundle: proof of corollary~\ref{c.central1D}}
The arguments for one-di\-men\-si\-onal invertible systems are classical and we only recall the main ideas.
Let $K$ be a compact invariant set endowed with a partially hyperbolic structure whose center bundle is $1$-dimensional and assume
that $K$ has no strong connection.
By Corollary~\ref{c.saddle}, the set $K$ is contained in a family of curves and circles
$\gamma_1,\dots,\gamma_k$ that are tangent at $E^c$ at points of $K$
and such that for any point $x\in \Gamma:=\cup_i\gamma_i$ close to $K$, the image of $x$
is still contained in $\Gamma$.
Any minimal subset $C$ of $K$ is either a periodic circle, or a periodic orbit,
or a Cantor set. In the third case, the orbit of any point $x\in \Gamma$ close to $C$ accumulates
on $\Lambda$ in the past or in the future. In particular there exist at most finitely many
non-periodic minimal sets and any orbit in $K$ accumulates in the future and in the past to
minimal sets.

One can $C^1$-approximate $f$ by a diffeomorphism $g$ whose periodic orbits are hyperbolic
and whose minimal sets are limit for the Hausdorff topology of periodic orbits.
By Corollary~\ref{c.robustness}, the maximal invariant set $\Lambda_g$ for $g$ in a neighborhood $U$ of $K$
is still contained in a one-dimensional $C^1$-submanifold $\Gamma_g$ and the dynamics of $g$ on $\Lambda_g$ satisfies the same properties
as $(K,f)$.
However, for any minimal set $C\subset \Lambda_g$, there exists a periodic orbit $O$ contained in an arbitrarily
small neighborhood of $C$. One deduces that $O$ is contained in $\Gamma_g$. Since the non-periodic
minimal sets are isolated in $\Gamma_g$ from the periodic orbits, hence cannot exist for $g$.
There are at most finitely many periodic orbits since they are hyperbolic.
This gives the conclusion of Corollary~\ref{c.central1D}.

\subsubsection{Two-dimensional center bundle: proof of Corollary~\ref{c.dim2}}
By Corollary~\ref{c.robustness}, for any diffeomorphism $g$ that is $C^1$-close to $f$
the maximal invariant set $\Lambda_g$ in $U$ has no strong connexion
and is contained in a locally invariant $C^1$-surface $\Si_g$.
Let us make two remarks:
\begin{itemize}
\item[--] $\Si_g$ is in general not a boundaryless compact manifold but some known results
for surface dynamics which only involve local arguments extend to this setting.
\item[--] If $h$ is a $C^1$-perturbation of the restriction $g_{|\Si_g}$ supported on  an arbitrarily small neighborhood of $\Lambda_g$,
then it extends as a diffeomorphism
of $M$ that is close to $f$ for the $C^1$-topology such that $\Si_g$ contains $\Lambda_h$
and is locally invariant in a neighborhood of this set.
Indeed, one can decompose $h=\varphi\circ g_{|\Si_g}$ where $\varphi$ is a diffeomorphism
of $\Si_g$ which is $C^1$-close to the identity and supported inside a small neighborhood of
$\Lambda_f$. Since $\varphi$ is isotopic to the identity among diffeomorphisms of $\Si_g$
close to the identity with compact support, one can extend $\varphi$ to a diffeomorphism of $M$
close to the identity.
\end{itemize}

Since $U$ is a filtrating set, the intersection $\cR(g)\cap U$ is a union of chain-recurrence classes $C$.
\begin{lemm}
Let us assume that the first case of the Corollary~\ref{c.dim2} does not hold.
Taking $g$ in a dense $G_\delta$-subset of $\cU$,
the center bundle over any non-trivial chain recurrence class $C\subset U$ has a \emph{dominated splitting}
$E^c_{|C}=E^c_1\oplus E^c_2$, i.e.
there exists $N\geq 1$ such that for any $x\in C$ and any $u,\in E^c_{1}(x)$, $v\in E^c_{2}(x)$ one has
$\|Dg^N.u\|\leq \frac 1 2 \|Dg^n.v\|$.
\end{lemm}
\begin{proof}
Taking $g$ in a dense $G_\delta$-subset of $\cU$,
one can assume that any chain-recurrence class $C$ which is not a periodic orbit
is limit of a sequence of hyperbolic periodic orbits $(O_n)$ (see~\cite{crovisier})
and then argue as in~\cite{PS1}.

A result by Pliss (see~\cite[theorem 2.1]{PS1}) asserts that by perturbation of $g_{|\Si_g}$ one can
turn one of the periodic orbits $O_n$ to be a saddle inside $\Si_g$.
By a standard Baire argument this implies that $C$ is also the limit of hyperbolic periodic orbits
whose stable and unstable spaces intersect $E^c$ along one-dimensional subspaces,
inducing an invariant splitting $E^c_1\oplus E^c_2$ of $E^c$ over the union of the orbits $O_n$.
If this splitting is not dominated one can create by perturbation of
$g_{|\Si_g}$ a homoclinic tangency for one of these saddles (see~\cite{gourmelon-tangence}) .
This perturbation may be extended
as a diffeomorphism of $M$ and the first case of the corollary holds.
Otherwise there exists a dominated splitting on the union of the $O_n$,
hence on their closure and on $C$ (see~\cite[Appendix B.1.1]{bdv}).
\end{proof}

As a consequence the set $\cR(g)\cap U$ decomposes into finitely many isolated periodic orbits
and a set whose center bundle has a dominated splitting.
The previous argument shows that (up to reduce $U$ and replace $f$ by a diffeomorphism $C^1$-close)
restrict to the case the center bundles of $\Lambda_f$ has a dominated splitting $E^c=E^c_1\oplus E^c_2$.
This also holds for any diffeomorphism $g$ in a neighborhood $\cU$.

By proposition~\ref{p.lissage}, one can consider $g\in \cU$ such that $\Si_g$ and
$g_{|\Si_g}$ are smooth. By perturbation of $g_{|\Si_g}$, one can furthermore assume that
all the periodic orbits in $\Si_g$ are hyperbolic and that there does not exist minimal sets in $\Si_g$
which is a finite union of circles that is normally hyperbolic.
One can now apply the result of~\cite{PS1}
(once again, the argument involves only the dynamics in a neighborhood of $K$
and the diffeomorphism $g$ may be only defined on a neighborhood of $K$).

\begin{theo*}[Pujals-Sambarino]
Consider a $C^2$-diffeomorphism $g$ of a surface $S$, and an invariant compact set
$K$ with a dominated splitting
$T_KS=E^c_1\oplus E^c_2$ such that $K$ does not contain any sink or source
and does not contain a minimal set which is a finite union of circles that is normally hyperbolic.
Then $K$ is a hyperbolic set.
\end{theo*}

For the dynamics of $g_{|\Si_g}$,
the set $\cR(g)\cap \Lambda_g$ is thus the union of sinks, of sources and of a saddle compact set $K$.
In particular the number of sinks and sources is finite. By Smale's spectral decomposition theorem
the set $K$ decomposes into finitely many transitive subsets as announced by Corollary~\ref{c.dim2}.

\subsection{Invariant foliations for surface hyperbolic sets: proof of Corollary~\ref{c.foliation}}
Let $f$ be a $C^2$-surface diffeomorphism and $K$ be an invariant compact set which is hyperbolic.
In particular $E^u$ is $2$-dominated by $E^s$ for $f^{-1}$ and
by Lemma~\ref{l.domination-cone},  there exists a continuous cone field $\cC$ of dimension $1$
which is $2$-contracted in a neighborhood of $K$: we have $E^u(x)\subset \cC(x)$
at each $x\in K$.

Let $\widehat M$ denotes the projectivization of the tangent bundle $TM$
(that is the Grassmannian bundle of $1$-dimensional tangent spaces)
and $p\colon \widehat M\to M$ the natural projection.
The tangent dynamics $Df$ induces a $C^1$-diffeomorphism $\widehat f$ of
$\widehat M$. Since the unstable bundle $E^u$ on $K$ is one-dimensional, it induces
a point $\widehat x$ in each fiber $p^{-1}(x)$ with $x\in K$,
defining a lift $\widehat K\subset \widehat M$ of $K$ which is invariant by $\widehat f$.
Since $x\mapsto E^u(x)$ is continuous, the set $\widehat K$ is compact.
By Proposition~\ref{p.lift2} and Remark~\ref{r.bunched},
there exist neighborhoods $\widehat U$ of $\widehat K$ and $U=p(\widehat U)$ of $K$
and a contracted continuous cone field $\widehat \cC$ of dimension $2$ on $\widehat U$,
that is transverse to $p$. By Lemma~\ref{l.domination-cone}, this proves that
$\widehat K$ has a dominated splitting $T_{\widehat K} \widehat M=\widehat E\oplus \widehat F$
where $\widehat F$ has $2$-dimensional spaces.
Since the tangent spaces to the fibers of $p$ at points of $\widehat K$ are preserved by
$\widehat f$ and since the fibers of $p$ are contracted by $\widehat f$
(see Proposition~\ref{p.fibre-contract}), one deduces that $\widehat K$
is partially hyperbolic with a dominated splitting $T_{\widehat K} \widehat M=\widehat E^{ss}\oplus \widehat E^c$.  The projection $Dp$ is an isomorphism between $\widehat E^c$ and
$T_KM$.

Since the fibers of $p$ are invariant by $\widehat f$
and tangent to $\widehat E^{ss}$ at points $\widehat x$ of $\widehat K$,
one deduces that each strong stable manifold $W^{ss}(\widehat x)$
is contained in $p^{-1}(p(x))$. In particular it intersects $\widehat K$
in a single point and the Main Theorem applies.

Let $S\subset \widehat M$ be a locally invariant $C^1$ surface containing $\widehat K$
and tangent to $\widehat E^c$ at points of $\widehat K$.
The projection $p\colon S\to M$ is a local diffeomorphism, injective on $\widehat K$,
hence injective on a neighborhood of $\widehat K$:
reducing $S$ if necessary, $p$ is a diffeomorphism between $S$ and a neighborhood
$U$ of $K$. Moreover $U$ is endowed with a $C^1$ line field
$\cL\colon x\mapsto p^{-1}(x)\cap S$ which is locally invariant by $Df$ by construction.
This line field uniquely integrates as a foliation $\cF^u$ on $U$ that is locally invariant
on a neighborhood of $K$ and that is tangent to $E^u$ at points of $K$.
In particular the leaves $\cF^u_x$ at points $x\in K$ contain the local stable manifolds
of $x$.

If $f$ is $C^r$, with $r>2$, then $\widehat f$ is $C^{r-1}$ and
$S$ can be chosen $C^{1,\alpha}$ for some $\alpha>0$
by Corollary~\ref{c.smooth}. In particular the line field $\cL$ and the foliation $\cF^s$ are $C^{1,\alpha}$.

% Notations
%$K\subset \Sigma$ with boundary, $E^c\oplus E^{uu}$, S
%$\lambda_K>1$.
%$\pi\colon T \to \Sigma_0$ $\Si_0$ open in $\Si$, $T$ open
%vertical vectors in $T$ tangent to $Ker(D\pi)$
%horizontal : tangent to $E$ (Extends $E^c$ continuously).
%$\cC^h_\beta$, $d_\pi$, $\Theta(\theta,z)$, $\Theta_\theta(z)$
%Constants:
%$\delta$= horizontal control
%$\lambda_0\in (1,\lambda_K)$
%$\eta$ vertical direction control
%$\beta$ open cone
%$\rho$ bounds expansion along fibers of multiplication $\Theta$
%$\gamma$: contraction constant
%$\bar \beta$ contraction cone $\beta'=\beta/\lambda_0$
%$C_\Si$, $c_f$
%$m>0$, $\varepsilon\colon (0,m_1)\to \RR^+_*$
%$\Si_m,\widehat Si_m$: open neighbd of $K$ in $\Si_0$.

\vskip 10pt

\begin{tabular}{l l l}
\emph{Christian Bonatti}
& \quad\quad \quad &
\emph{Sylvain Crovisier}
\medskip\\

Institut de Math\'ematiques de Bourgogne
&& Laboratoire de Math\'ematiques d'Orsay\\
CNRS - URM 5584
&& CNRS - UMR 8628\\
Universit\'e de Bourgogne
&& Universit\'e Paris-Sud 11\\
Dijon 21004, France
&& Orsay 91405, France.
\end{tabular}

\end{document}